\documentclass[12pt]{article}
\usepackage{graphicx} 
\usepackage{amsmath}
\usepackage{amsthm}
\usepackage{amssymb}
\usepackage{mathtools}
\usepackage{bussproofs}

\usepackage{array}
\usepackage{tikz}
\usetikzlibrary{positioning,quotes}

\newtheorem{theorem}{Theorem}[section]
\newtheorem{corollary}[theorem]{Corollary}
\newtheorem{lemma}[theorem]{Lemma}
\newtheorem{claim}[theorem]{Claim}
\newtheorem{proposition}[theorem]{Proposition}

\theoremstyle{definition}
\newtheorem{definition}[theorem]{Definition}

\def\fCenter{\Rightarrow}

\def\liff{\leftrightarrow}

\newcommand{\N}{\mathbb{N}}

\newcommand\TN{{\hbox{\bfseries \tt N}}}

\newcommand\dts{,\dots ,}

\newcommand\be{\begin{equation}}
\newcommand\bel[1]{\begin{equation}\label{#1}}
\newcommand\ee{\end{equation}}
\newcommand\ba[1]{\begin{array}{#1}}
\newcommand\ea{\end{array}}
\newcommand\bl{\begin{lemma}}
\newcommand\bll[1]{\begin{lemma}\label{#1}}
\newcommand\el{\end{lemma}}
\newcommand\bi{\begin{itemize}}
\newcommand\ei{\end{itemize}}
\newcommand\bt{\begin{theorem}}
\newcommand\btl[1]{\begin{theorem}\label{#1}}
\newcommand\et{\end{theorem}}
\newcommand\ben{\begin{enumerate}}
\newcommand\een{\end{enumerate}}
\newcommand\bpr{\begin{proposition}}
\newcommand\bprl[1]{\begin{proposition}\label{#1}}
\newcommand\epr{\end{proposition}}
\newcommand\bcl{\begin{claim}}
\newcommand\ecl{\end{claim}}
\newcommand\bprf{\begin{proof}}
\newcommand\eprf{\end{proof}}
\newcommand\bdf{\begin{definition}}
\newcommand\edf{\end{definition}}

\def\UnivTk#1#2{{T^{U,#1,#2}}}
\def\Univk#1#2{{U^{#1,#2}}}
\def\UnivT#1{{T^{U,#1}}}

\def\Tbeta{\beta}
\def\TPr{\pi}

\def\Tparent{{\mathchoice{\hbox{\it Parent}}{\hbox{\it Parent}}{\hbox{\scriptsize \it Parent}}{\hbox{\scriptsize \it Parent}}}}

\def\TTm{{\mathchoice{\hbox{\it Time}}{\hbox{\it Time}}{\hbox{\scriptsize\it Time}}{\hbox{\scriptsize\it Time}}}}
\def\MSP{{\hbox{\sc msp}}}

\def\ThTG{{\mathcal T}^{\hbox{\tiny \em G}}}
\def\ThTP{{\mathcal T}^{\hbox{\tiny \em P}}}
\def\ThTPf{{\mathcal T}^{\hbox{\tiny \em P,f}}}
\def\PV{\mathrm{PV}}

\title{Herbrand Game Complexity}
\author{Sam Buss \\
Department of Mathematics \\
Univ.\ of California, San Diego \\
sbuss@ucsd.edu
\and 
Pavel Pudl\'ak \thanks{partially supported by GA\v{C}R grant 25-16311S and institute's grant RVO~67985840.}\\
Institute of Mathematics \\
Czech Academy of Sciences, Prague \\
pudlak@math.cas.cz}
\date{\today}

\begin{document}

\maketitle

\begin{abstract}
The Student-Teacher game is an extension of Herbrand's theorem. In the game, 
Student and Teacher take turns giving values for existentially and universally quantified variables, 
and Student is allowed to backtrack to propose other values for existentially quantified variables.
The game
has become increasingly important for proving lower
bounds on provability in theories of bounded arithmetic.
In those applications, the game is played in an arithmetical theory; however, 
this paper studies the Student-Teacher game in the setting
of pure first-order logic, so it is more closely related to Herbrand's theorem
and the midsequent theorem.
When played in pure first-order logic, 
a prenex formula~$\varphi$ is logically valid if and only if there is a Student-Teacher
game with a winning strategy for Student for establishing~$\varphi$. 
We present a refined version of the Student-Teacher game in arbitrary first-order
universal theories and include a proof of the validity of Student-Teacher games
from the sequent calculus midsequent theorem in an appendix.

The game is presented as a finite tree with vertices and edges labeled by terms, 
and with a total order on the nodes. The totally ordered tree represents the players' interaction. 
Our main results show that minimal trees in the Student-Teacher game can be arbitrarily complex. 
Specifically, for every totally ordered tree~$T$, 
we construct a valid prenex formula~$\varphi$ such that 
every winning Student-Teacher game for~$\varphi$ contains $T$ as a substructure.   
Using formulas parameterized by numerals, we prove similar results
for computable and partial computable sequences of totally ordered trees.
It follows not only that there is no computable bound on the size of Herbrand disjunctions, 
which is a well-known fact, but also that there is no computable bound on their complexity
in the sense that it is not possible to restrict the types of totally ordered trees 
knowing only the length of the formula.
\end{abstract}

\section{Introduction}

Herbrand's theorem is a useful way to represent proofs in first-order logic and 
to interpret proofs constructively. 
Suppose $\exists x_1\, \exists x_2 \cdots \exists x_k A(x_1, \ldots, x_k)$ is a
valid, purely existential first-order formula in the language~$L$ where
$A$~is quantifier-free. For simplicity,
assume that the equality symbol~$=$ is not included in~$L$. 
Herbrand's theorem tells us that
there is a finite list of vectors~$\vec t_i$ of $L$-terms such that the disjunction
$\bigvee_i A(\vec t_i)$ is a tautology. 

More generally, suppose $C$ is a prenex
formula
\begin{equation}\label{eq:HerbrandExis2}
\forall y_0\, \exists x_1\, \forall y_1\, \exists x_2 \cdots \forall y_{k-1}\, \exists x_k \,
   B(y_0,x_1,y_1,x_2,\ldots,y_{k-1},x_k)
\end{equation}
where $B$ is quantifier-free.  By introducing Skolem functions~$g_i(x_1,\ldots, x_i)$,
the formula~$C$ is valid if and only if the purely existential formula
\begin{equation}\label{eq:HerbrandExis3}
\exists x_1\, \exists x_2\, \exists x_3 \cdots \exists x_k \, 
  B(g_0, x_1, g_1(x_1), x_2, g_2(x_1, x_2), x_3,\ldots, g_{k-1}(x_1, \ldots, x_{k-1}), x_k)
\end{equation}
is valid.\footnote{The 0-ary function~$g_0$ is a new constant symbol.
Each $g_i$ is a new $i$-ary function symbol.}   
Herbrand's theorem applied to this existential formula
implies there is a finite set of terms $t_{i,j}$, 
where $1\le j\le k$ and $1\le i\le I$ for some~$I$,
such that
\begin{equation}\label{eq:Herbrand4}
\bigvee\nolimits _{i\le I} B( g_0, t_{i,1}, g_1(t_{i,1}), t_{i,2} , g_2( t_{i,1}, t_{i,2} ),
      \ldots, g_{k-1} ( t_{i,1}, t_{i,2}, \ldots, t_{i,k-1} ), t_{i,k} )
\end{equation}
is a tautology.
Although the formulation of the disjunction~(\ref{eq:Herbrand4}) may look superficially simple, the
terms~$t_{i,j}$ and the relationships between them may be very complex. 
The aim of the present paper
is to understand how complex this can be.

The complexity of the Herbrand disjunction will be studied by
using a game formulation of Herbrand's theorem, called the ``Student-Teacher'' game, 
that has been used in the past to extract additional 
constructive content from Herbrand's theorem. The modern versions
of the Student-Teacher game arose 
from Kraj\'\i\v cek-Pudl\'ak-Takeuti~\cite{KPT:BAandPH} showing 
conditional independence results in bounded arithmetic.
A winning strategy for Student in a Student-Teacher game is a special 
case of a Herbrand proof of the prenex formula~(\ref{eq:HerbrandExis3}); hence the winning
strategy can be viewed as a first-order proof of~(\ref{eq:HerbrandExis2}). 
Buss-Ko{\l}odziejczyk-Thapen~\cite{BKT:fragments} and
Li-Oliveira~\cite{LiOliveira:Unprovability} have given general
proofs of the completeness and soundness of the Student-Teacher game 
as a proof system for prenex formulas. Appendix~\ref{sec:midsequent}
below will give another proof based on the midsequent theorem for first-order logic.
The midsequent theorem states that every valid prenex formula has a cut-free sequent calculus
proof in which the top part uses only propositional inferences and the bottom
part uses only quantifier inferences and contractions. 
Appendix~\ref{sec:midsequent} shows how to construct a
winning strategy in the Student-Teacher game from such a proof.%
\footnote{Midsequents and Herbrand disjunctions are essentially the same objects. So, we can also say that we study the complexity of midsequents.}

The idea of the game goes back to
Kreisel~\cite{Kreisel:interpretationI,Kreisel:interpretationII} 
where he introduced the \emph{no-counterexample interpretation}. 
In that interpretation, the universal quantifier is interpreted 
as the non-existence of a counterexample to the negation. 
He also observed that a sentence of the form 
$\forall x\,\exists y\,\forall z\,\varphi(x,y,z)$ is logically valid 
if and only if for some $n$, there are terms
\bel{e-terms}
t_1(x),\,t_2(x,z_1),\,t_3(x,z_1,z_2),\,\ldots,\,t_n(x,z_1,\ldots,z_{n-1}),
\ee
where the displayed variables are the only variables that can occur in the term, such that
\bel{e-dis}
\varphi(x,t_1(x),z_1)\vee\varphi(x,t_2(x,z_1),z_2)\vee\dots\vee\varphi(x,t_n(x,z_1,\dots,z_{n-1}),z_n)
\ee
is a tautology. 
  
The formula (\ref{e-dis}) was applied 
by Kraj\'\i\v cek-Pudl\'ak-Takeuti~\cite{KPT:BAandPH} 
to prove separations 
of theories in bounded arithmetic. This method is
known as the 
KPT theorem, and a number of authors have applied it 
to give separation or unprovability results (often conditional)
for fragments of bounded arithmetic, including
Pudl\'ak~\cite{Pudlak:somerelations},
Buss~\cite{Buss:BAandPH},
Zambella~\cite{Zambella:notes},
Cook-Thapen~\cite{CookThapen:replacement},  
Cook-Kraj\'\i{\v c}ek~\cite{CookKrajicek:NPsubsetPpoly}, 
Byd{\v z}ovsk{\'y}-Kraj{\'\i\v c}ek-Oliveira~\cite{BKO:Consistency}, 
Pich-Santhanam~\cite{PichSanthanam:StrongCoNondet},   
Carmosino-Kabanets-Kolokolova-Oliveira(-Tsintsilidas)~\cite{CKKO:LEARN,CKKOT:Hierarchy}, 
Li-Oliveira~\cite{LiOliveira:Unprovability},
Je\v{z}il~\cite{Jezil:Factorization}, and
Je\v{z}il-Tsintsilidas~\cite{JezilTsintsilidas:StudentTeacher}.  
(The name ``KPT Theorem'' comes from~\cite{Buss:BAandPH}.)

The intuition behind the Student-Teacher interpretation 
of the disjunction~(\ref{e-dis}) is as follows. 
Teacher plays a value~$x$, and the 
task of Student is to find a~$y$ such that 
$\forall z\,\varphi(x,y,z)$ is true. 
(The word ``true'' should not be interpreted literally, 
since we do not talk about a concrete model.) 
Student proposes $t_1(x)$. 
If it is a correct value for~$y$, the game ends; 
otherwise, Teacher responds with a counterexample~$z_1$, 
which Student can use to propose a 
value $t_2(x,z_1)$ for~$y$, and so on. Since the disjunction~(\ref{e-dis}) is a tautology, 
in any model one of Student's values should be correct. 

In the case of the prefix $\forall^*\exists^*\forall^*$, 
we have the standard form described above,\footnote{Stars ($^*$'s) 
denote that the same quantifier can be repeated, or missing.} 
but in general it is not true that a simple standard form exists. 
Already for the prefix $\exists\forall\exists$, there is no simple standard form. 
Consider possible games for a sentence $\exists x\,\forall y\,\exists z\,\varphi(x,y,z)$. 
Let Student propose $a_1,\ldots, a_n$ 
as possible values for~$x$ and 
Teacher respond with $b_1,\ldots, b_n$ as (purported) counterexamples 
for the values of~$y$. 
For each $i=1,\ldots, n$, Student can use $m_i$~many plays to
propose multiple values 
$c_{i,1},\ldots, c_{i,m_i}$ for the existentially quantified~$z$.
Thus in the $\exists\forall\exists$ case, the game  
can be viewed as being played on a tree of depth~2. 
In this tree the root is unlabeled, 
edges are labeled with values played by Student, 
the internal nodes are labeled with values played by Teacher, 
and the leaves are unlabeled. 
In the purely existential case above,
the shape of the tree of depth~1 can be specified by the single parameter~$n$. 
In the $\exists\forall\exists$ case however,
there are a variety of possible shapes for the game trees, given by parameters $n,m_1,\ldots, m_n$.
 
These trees, by themselves, do not totally capture the complexity of the game,
because Student can play $a_1,\ldots, a_n$, $c_{1,1},\ldots, c_{1,m_1},\ldots, c_{n,1},\ldots, c_{n,m_n}$ 
in any order except that she can only play $c_{i,j}$ after 
$a_i$ and~$b_i$ have been played. 
In other words, the order of plays by Student must be consistent with the ordering on branches, 
where the ordering is from the root to the leaves.

We will use \emph{totally ordered trees} (see Definition~\ref{def:totallyOrderedTree}) 
to describe the order in which Student makes plays
in the Herbrand game. These trees are our main conceptual contribution; 
they enable us to capture the structural complexity of Herbrand disjunctions.

An example of such trees is given in Figure~\ref{fig:treeExample}
for a formula 
\[
\forall y_0 \, \exists x_1 \, \forall y_1 \, \exists x_2 \, \forall y_2 \, \varphi(y_0,x_1,y_1,x_2,y_2) .
\]
The nodes of the tree are labeled with values played by Teacher; the
edges are labeled with values played by Student. In the figure,
$b$~labels the root and is the value for~$y_0$ played by Teacher. The
edges labeled $a_1,a_2,a_3$ represent three plays of
Student proposing values for the existentially quantified~$x_1$.
The values $b_1,b_2,b_3$ give three values played by Teacher
as values for~$y_1$ as responses to the
corresponding Student's plays. The deepest edges and nodes are
similarly labeled with values $a_{1,1},a_{1,2},a_{2,1},a_{3,1}$ proposed 
by Student for~$x_2$
and values $b_{1,1},b_{1,2},b_{2,1},b_{3,1}$ played by Teacher for~$y_2$ in response.  
Student wins
the game if the disjunction of the instantiations of~$\varphi$ in all four
branches, namely
\[
\varphi(b,a_1,b_1,a_{1,1},b_{1,1}) \lor
   \varphi(b,a_1,b_1,a_{1,2},b_{1,2}) \lor
   \varphi(b,a_2,b_2,a_{2,1},b_{2,1}) \lor
   \varphi(b,a_3,b_3,a_{3,1},b_{3,1}) 
\]
is a tautology. The Student-Teacher game generalization of
Herbrand's theorem implies that the prenex formula
is logically valid if and only if there is a totally ordered tree for which
Student has a winning strategy.

In this characterization, we can assume that Teacher always plays new free variables
(for the $b$-variables), and that Student must play
terms that involve as free variables only variables already played by Teacher. In the sequel, we will assume both conditions. 
Thus the variables that can appear in Student's terms depend on
Student's order of play in the game tree. Figure~\ref{fig:treeExample} shows
this for two different orderings of the game tree, namely a breadth-first play order
and a depth-first play order. Of course, other play orders are possible.

Formal definitions for Herbrand game trees are given in the next section, but
the highlights are as follows. The
Herbrand game will be played for a prenex formula
of the form
\[
\forall \vec y_0\,\exists \vec x_1\,\forall \vec y_1\cdots\exists \vec x_k\,\forall \vec y_k\,\varphi(\vec y_0,\vec x_1,\vec y_1,\dots,\vec x_k,\vec y_k),
\]
with $\varphi$ quantifier-free. Teacher plays values~$\vec b_i$ for the
variables~$\vec y_i$, and Student plays values~$\vec a_i$ for the~$\vec x_i$.
\begin{description}
\item[1.] \label{page:HerbrandTreeProperties} 
Teacher and Student play values starting from the
root of the Herbrand game tree, working towards the leaves.
Teacher plays new free variables as values, and they label
nodes in the tree.
Student always plays terms as values, and these label edges.
Student plays terms that involve only the free variables
previously played by Teacher. 
\item[2.] Student always plays values~$\vec a_i$ 
for an entire existential block~$\exists \vec x_i$ at once. Teacher then responds on 
the edge just played by Student
by giving values~$\vec b_i$ (free variables) for the entire universal block~$\vec y_i$.
\item[3.] Student never repeats a query; i.e., Student never
re-plays the same values {\em below the same node} in the
game tree. 
\item[4.] The first and last blocks of universally quantified
variables, $\vec y_0$ and~$\vec y_k$, may be vacuous. If so,
Teacher makes null plays for these values. If $\forall \vec y_0$ is vacuous,
we assume w.l.o.g.\ that the language contains a constant symbol so
Student can play a closed term for her first move.
\item[5.] If the innermost block $\forall y_k$ of quantifiers is
vacuous, Student does not play values for the innermost $\exists x_k$
values until the rest of the Herbrand game tree has been completely played.
The reason for this is that Student will receive no Teacher answer
for plays of values for the $\vec x_k$'s, thus she can freely postpone these
plays until the end of the game.
\end{description}
Our proofs will involve games with an additional \emph{copycat} restriction on plays by
Student. In a copycat strategy, the value played by Student is
always equal to the previous value played by Teacher. Specifically,
after Teacher plays a value~$b$ on some node in the game tree,
Student picks some node (not necessarily the same node where Teacher just played!) and
plays the same value~$b$
as a label on a new outgoing edge of that node.

\begin{figure}[t]
\begin{center}
\begin{tikzpicture}[Teach/.style=fill,draw,circle,inner sep = 1.5pt]
\node[Teach,label={175:$b$}] (root) at (0.5,0) {};
\node[Teach,label={177:$b_1$}] (b1) at (-1, -1) {};
\node[Teach,label={0:$b_2$}] (b2) at (1.5, -1) {};
\node[Teach,label={0:$b_3$}] (b3) at (3, -1) {};
\node[Teach,label={180:$b_{1,1}$}] (b11) at (-1.5, -2.5) {};
\node[Teach,label={0:$b_{1,2}$}] (b12) at (-0.5, -2.5) {};
\node[Teach,label={0:$b_{2,1}$}] (b21) at (1.5, -2.5) {};
\node[Teach,label={0:$b_{3,1}$}] (b31) at (3, -2.5) {};
\draw (root) --node[above left]{$a_1$} (b1) --node[left]{$a_{1,1}$} (b11);
\draw (b1) --node[right]{$a_{1,2}$} (b12);
\draw (root) --node[below left]{$a_2$} (b2) --node[right]{$a_{2,1}$} (b21);
\draw (root) --node[above]{$a_3$} (b3) --node[right]{$a_{3,1}$} (b31);
\end{tikzpicture}
\end{center}
\caption[A Herbrand game tree]{A Herbrand game tree (without ordering).  \small
The nodes are labeled with
$b$-values, denoting distinct free variables played by Teacher.
The edges are labeled with $a$-values, denoting terms played by Student. The
dependence of Student's terms on Teacher's variables is controlled
by the total ordering of the nodes.
\parindent = 2em

For example, the breadth-first order (conflating nodes with their labels) is
$ b \prec b_1 \prec b_2 \prec b_3 \prec b_{1,1} \prec b_{1,2} \prec b_{2,1} \prec b_{3,1}$. 
The $a$-terms may depend on variables as 
$a_1 = a_1(b)$, 
$a_2 = a_2(b,b_1)$,
$a_3 = a_3(b,b_1, b_2)$,
$a_{1,1} = a_{1,1}(b,b_1, b_2, b_3)$,
$a_{1,2} = a_{1,2}(b,b_1, b_2, b_3, b_{1,1})$,
$a_{2,1} = a_{2,1}(b,b_1, b_2, b_3, b_{1,1}, b_{1,2})$, and
$a_{3,1} = a_{3,1}(b,b_1, b_2, b_3, b_{1,1}, b_{1,2},b_{2,1})$.

As a second example, the depth-first order 
is
$b \prec b_1 \prec b_{1,1} \prec b_{1,2} \prec b_2 \prec b_{2,1} \prec b_3 \prec b_{3,1}$,
and allows the dependencies 
$a_1 = a_1(b)$, 
$a_{1,1} = a_{1,1}(b,b_1)$,
$a_{1,2} = a_{1,2}(b,b_1,b_{1,1})$,
$a_2    = a_2   (b,b_1,b_{1,1},b_{1,2})$,
$a_{2,1} = a_{2,1}(b,b_1,b_{1,1},b_{1,2},b_2)$,
$a_3    = a_3   (b,b_1,b_{1,1},b_{1,2},b_2,b_{2,1})$, and
$a_{3,1} = a_{3,1}(b,b_1,b_{1,1},b_{1,2},b_2,b_{2,1},b_3)$.
}
\label{fig:treeExample}
\end{figure}

\paragraph{The Herbrand game in universal theories.}
So far we have assumed that we only use pure predicate logic without equality.
To handle equality, we need to work over a theory that includes equality axioms,
formulated as universal formulas. More generally, we can work with
arbitrary theories that are axiomatized by purely universal sentences. 
Very little changes for Herbrand's theorem and the Student-Teacher
games when working relative to a universal theory.

Let $\mathcal T$ be a universal theory and let $\Phi$ be a 
prenex sentence provable in~$\mathcal T$. 
Then there is a finite set of (universal) axioms $A=\{A_1\dts A_k\}$ of~$\mathcal T$ such that 
the sequent $A\Rightarrow \Phi$ is provable. According to the midsequent theorem, 
there therefore exists a sequent 
\[
\Gamma\Rightarrow\Delta
\]
that is provable in propositional logic, 
where $\Gamma$ consists of substitution instances of the matrices, the quantifier-free parts, of the 
axioms $A_1\dts A_k$, and $\Delta$ consists of substitution instances 
of the matrix of $\Phi$, and such that $A \Rightarrow \Phi$
is derivable from the sequent $\Gamma \Rightarrow \Delta$
by a proof~$P$ that uses only quantifier and structural inferences.
Appendix~\ref{sec:midsequent} describes how to construct a Herbrand game tree 
with a winning strategy
for Student via the midsequent theorem. This construction can 
also be applied
to the proof~$P$. Indeed, without loss of generality, the lower part
of the proof~$P$ below the midsequent
starts with $\forall$:left inferences 
that derive $A \rightarrow \Phi$ from $\Gamma\Rightarrow \Phi$, 
and the rest of $P$ is handled by the construction
in Appendix~\ref{sec:midsequent}.

The important difference when working over a universal theory~$\mathcal T$
is that now $\bigvee \Delta$ may not be a propositional tautology;
instead, it is provable in~$\mathcal T$.

Our second bound on the complexity of Herbrand disjunctions 
in Theorem~\ref{thm:mainPsiN}
treats universal theories by a slightly different method.
The provable sequent $A\Rightarrow \Phi$ is handled with a Student-Teacher
game for a prenexification of the provable formula $\bigwedge A\rightarrow\Phi$. 
Specifically, suppose $\Phi$ is 
\[
\forall \vec y_0\,\exists \vec x_1\,\forall \vec y_1\cdots\exists \vec x_k\,\forall \vec y_k\,\varphi(\vec y_0,\vec x_1,\vec y_1, \dots,\vec x_k,\vec y_k).
\]
and $\mathcal T \vdash \Phi$.
Let each $A_i \in A$ have the form $\forall \vec v\, A_i^*(\vec v)$
with $A^*_i$ quantifier-free.  
We then replace the matrix~$\varphi$ of~$\Phi$ with the 
prenex formula 
\begin{equation}\label{eq:AddingAxioms}
    \exists \vec v \, [\varphi(\vec y_0,\vec x_1,\vec y_1,\ldots,\vec x_k,\vec y_k)
      \lor \bigvee\nolimits_i \lnot A^*_i(\vec v) ] .
\end{equation}
The resulting formula is valid and is in prenex form with $2k+2$ quantifier blocks.\footnote{We
could reduce the quantifier alternation by one by merging the $\exists \vec v$
quantifiers with the $\exists x_k$ quantifier block. This avoids increasing the
alternation of quantifiers, but would have the disadvantage of complicating other
aspects of our constructions.}

The methods of working with universal theories are entirely equivalent,
and both methods are used below. 
Theorem~\ref{thm:firstLBproof} uses the first method
for simplicity; Theorem~\ref{thm:mainPsiN}
uses the second method.

We remark that one can get important information also from provability in theories 
that are not universal. An early example is a version of the KPT theorem in~$S^1_2$, 
introduced by the second author~\cite{Pudlak:somerelations}. Subsequent work has exploited
similar constructions to prove independence results for fragments of bounded
arithmetic; for these, see
\cite{Buss:BAandPH,%
Zambella:notes,%
CookThapen:replacement,%
CookKrajicek:NPsubsetPpoly,%
BKO:Consistency,%
PichSanthanam:StrongCoNondet,%
CKKO:LEARN,%
CKKOT:Hierarchy,%
LiOliveira:Unprovability}.

\paragraph{Some related work.} 
We will only mention some results that connect the complexity of 
Student-Teacher games with first-order theories. 

In the seminal paper \cite{KPT:BAandPH}, the authors showed that the 
sentences of the form $\forall x\exists y\forall z\varphi(x,y,z)$ provable in $PV_1$, 
where the quantifiers are bounded and $\varphi$ is a sharply bounded formula, 
can be proved using Student-Teacher games with a constant number of rounds, 
but constant numbers of rounds do not suffice for such theorems of~$S^1_2$, 
unless $\mathrm{NP}\subseteq \mathrm{P/poly}$. 
This was then used to separate $PV_1$ from $S^1_2$ under the same assumption. 
On the other hand, Student-Teacher games can be used
for theorems of~$S^1_2$  
if a polynomial number of rounds is allowed~\cite{Pudlak:somerelations}.

Kraj\'\i\v{c}ek, Pudl\'ak, and Sgall studied classes of TF$\Sigma_2$ defined by Student-Teacher games with a bounded number of rounds~\cite{KPS:interactive}. They proved that for any sublinear, unbounded, increasing, polynomial-time
function $r(n)$, the class defined by games with $r(n)+1$ rounds is not contained in the class defined by games with $r(n)$ rounds.

This hierarchy was refined in Je\v zil~\cite{Jezil:Factorization} 
and Je{\v z}il-Tsintsilidas~\cite{JezilTsintsilidas:StudentTeacher}. 
They introduced  \emph{parallel queries}, which is a slightly misleading terminology because it is Student's \emph{answers} that are parallel. What it means is that given Teacher's question~$b$, 
Student replies with multiple answers $a_1,\dots,a_p$, and Teacher produces counterexamples $b_1,\dots,b_p$ only after Student presents all these answers. In other words, Student cannot use any of the  $b_1,\dots,b_p$ to compute $a_1,\dots,a_p$. 
Je{\v z}il-Tsintsilidas~\cite{JezilTsintsilidas:StudentTeacher}
showed that any polynomial number of parallel answers cannot replace one additional round, unless $\mathrm{NP}\subseteq\penalty10000 \mathrm{P/poly}$.\footnote{This is the basic form of their theorem, which in general applies to higher total function classes.} Using this theorem they show that the sequence of theories
\[
PV_1,\ PV_1 + BB(\Sigma^b_1 ),\ PV_1 + LLIND(s\Sigma^b_1),\ PV_1 + LLIND(\Sigma^b_1),\ S^1_2
\]
is strictly increasing unless $\mathrm{NP}\subseteq \mathrm{P/poly}$. Before, these separations were only known under stronger assumptions. 

Parallel answers cannot be expressed using games as defined in this paper because we require that Teacher replies immediately. To be able to do it, we would need to define the total ordering on nodes \emph{and edges,} and not restrict Teacher.

\paragraph{Outline of the paper and main results.}
Section~\ref{sec:HerbrandGameTree} gives the complete definitions of the Student-Teacher game
and Herbrand game
trees, including the notions of totally ordered trees, embeddings between trees,
Herbrand game trees,
and winning strategies
for Student.  We also describe an alternative way of viewing the
Herbrand game tree as a matrix, where the matrix entries are plays by Teacher
and Student, and the columns correspond to branches in the game tree.
Appendix~\ref{sec:midsequent} gives a direct proof, from the
midsequent theorem, that a prenex formula~$\Phi$ is logically valid 
if and only if it has a Herbrand game tree~$T_\Phi$ that
admits a winning strategy for Student.

Our main theorems state that there is 
no computable bound (as a function of~$\Phi$) on the complexity
of Herbrand game trees~$T_\Phi$ needed for~$\Phi$.
In effect, Herbrand trees~$T_\Phi$
that admit a winning strategy for Student
may need to be arbitrarily complicated, including with
complicated orderings~$\prec$. It is a folklore result
that the Herbrand theorem for $\forall \exists$ formulas~$\Phi$ does
not have a computable bound on the number of terms (see~\cite{Buss:herbrandtheorem}).
Specifically, our results apply to arbitrary prenex formulas~$\Phi$; and 
state that the
order in which Student and Teacher play labels in the Herbrand game tree 
can be as complicated as any computably enumerable set.

The first result, discussed in Section~\ref{sec:arbitrary}, 
is that for \emph{every} 
totally ordered tree~$T$,
there is a prenex formula~$\Phi_T$ such that any Herbrand
game tree with a winning strategy for Student must be 
at least as complex as~$T$.
Conversely, we show that the tree~$T$ can be used to give a winning
strategy for Student, except with the proviso that extra edges are added
to leaves of~$T$, allowing Student a final round of plays.  The presence
of the extra edges corresponds to the inclusion of an additional
innermost block of existential quantifiers when working with 
provability relative to a universal theory~$\mathcal T$. This is due
to the construction of adding existential quantifiers $\exists \vec v$ as
shown in~(\ref{eq:AddingAxioms}).

Since the construction of~$\Phi_T$ is done for an arbitrary~$T$,
information-theoretic considerations imply that, in general, 
$\Phi_T$~must be as large as~$T$. Our next results
improve on this, by constructing a family of logically valid, prenex
formulas~$\Phi$ such that any Herbrand game tree~$T$ with a winning
strategy for Student must have size that is {\em not}
computably bounded in the size of~$\Phi$. Thus, there is no
computable bound on the size of a Herbrand game tree as a function of
the size of a formula~$\Phi$.  This strengthens the folklore fact that
there is no computable upper bound on
the minimal number of terms in a Herbrand disjunction~\cite{Buss:herbrandtheorem}.
This is a strengthening since 
the construction of the Herbrand game tree implies that not only must
the number of terms in a Herbrand disjunction be large in the worst
case, the total ordering of
the Herbrand game tree must be highly complex. In fact, what
we show is that, for any computable sequence of trees $T_1,T_2,T_3,\ldots$, 
there are formulas $\Phi_1,\Phi_2,\Phi_3, \ldots$ such that any winning strategy for Student 
for~$\Phi_r$ must be at least as complex
as~$T_r$.  
Even more, for any partial computable sequence of trees~$T_r$,
i.e.\ where the set of pairs $(r,T_r)$ is computably enumerable, there
are formulas~$\Phi_r$ such that any winning strategy for Student 
for~$\Phi_r$ must be at least as complex
as~$T_r$.  The formulas~$\Phi_r$ can be constructed in polynomial time
as a function of~$r$.

Sections \ref{sec:FirstLBproof} and~\ref{sec:SecondLBproof} prove this
in two different ways.  The proof in Section~\ref{sec:SecondLBproof}
proves this with a direct, somewhat ``brute force'', approach; it works in
pure first-order logic and forms formulas~$\Phi$
that contain a single extra innermost block of existential quantifiers; these
extra existential quantifiers are from the construction shown in~(\ref{eq:AddingAxioms}).
The proof in Section~\ref{sec:FirstLBproof} works relative to a universal theory,
and uses a construction
of the formulas~$\Phi$ that adds 
two extra blocks of quantifiers,~$\exists\forall$.  
(It would need to add three blocks,
$\exists \forall \exists$, if the same construction 
were rephrased over pure first-order logic instead of
over a universal theory).  For simplicity, Section~\ref{sec:FirstLBproof}
uses a computable sequence of trees~$T_r$ instead of a partial computable 
sequence.
Although it adds
more quantifiers, an interesting aspect of
the proof of Section~\ref{sec:FirstLBproof} is that the formulas~$\Phi_r$ express
a form of self-consistency for Herbrand game provability,  
which may point toward further results.

Appendix~\ref{sec:UnivTrees} gives a method for constructing universal depth~$k$
totally ordered trees. Since any totally ordered depth~$k$ tree can be embedded into
one of these trees with only a polynomial size blowup, these
trees provide uniform trees that can suffice as Herbrand game trees for
any logically valid formula.

\paragraph{Student-Teacher game in concrete structures.}
It is also possible to use the Student-Teacher game in a concrete structure, 
say, in the standard model of the integers, although we will not consider this in this paper.
Note that in a concrete structure the game would be trivial for true sentences 
if we do not restrict the possible strategies of Student: 
If Student is allowed to play optimally, the game finishes after $k$ plays if there are $k$ quantifier blocks. 
If, however, the sentence is provable in a universal theory~$\mathcal T$,
we know that there exists a strategy based on terms available in~$\mathcal T$. 
For example, let {\it TruePV} denote the set of all true universal sentences 
in the language of $\PV$ of polynomial time computable functions.
Provability of a sentence in {\it TruePV} 
implies that Student has a strategy defined using $\PV$ terms. 
Since $\PV$ terms define polynomial time computable functions in 
$\N$, in this structure Student has a strategy defined using 
polynomial time computable functions. 

The first paper introducing the KPT construction,~\cite{KPT:BAandPH}, can be
viewed as working in a concrete structure, namely, the integers. Similarly
for the already-mentioned subsequent works
\cite{Pudlak:somerelations,%
Buss:BAandPH,%
Zambella:notes,%
CookThapen:replacement,%
CookKrajicek:NPsubsetPpoly,%
BKO:Consistency,%
PichSanthanam:StrongCoNondet,%
CKKO:LEARN,%
CKKOT:Hierarchy,%
LiOliveira:Unprovability}
proving independence
results (mostly conditional) for bounded arithmetic. 

The present paper only studies the Student-Teacher game 
as an interpretation of Herbrand's theorem. However, 
in concrete structures, one can also consider generalizations 
that do not correspond directly to Herbrand games. 
One of these is the game where the strategy of Student is not oblivious. 
This means that knowing the value played by Teacher, 
Student can evaluate the atomic formulas in the structure,
and adaptively choose which node to play at and which term to play.

\section{The Herbrand game tree}\label{sec:HerbrandGameTree}

An ordered tree is a tree in which the children of each
node are ordered, i.e., there is a first child, second child, etc.
Trees that describe the order of play of Herbrand games have {\em all}
the nodes in the tree ordered, reflecting the order
of Student's moves. We call such trees \emph{totally ordered}:

\begin{definition}\label{def:totallyOrderedTree}
A \emph{totally ordered tree}~$T$ is a finite rooted tree with a total order~$\prec$ on 
the set of all nodes of~$T$ such that for all nodes $u$ and~$v$ with
$v$ a child of~$u$, we have $u\prec v$.
\end{definition}

\begin{definition}\label{def:embedding}
Let $S$ and~$T$ be totally ordered trees. An \emph{embedding} of $S$ into~$T$
is an injective mapping~$\tau$ from $S$ to~$T$
such that 
\begin{itemize}
\item $\tau$ maps the root of~$S$ to the root of~$T$, and
respects the total orders
of $S$ and~$T$; i.e., $u \prec_S v$ holds if 
and only if $\tau(u) \prec_T \tau(v)$ holds.
\item $\tau$ respects the parent/child relation;
i.e., $u$~is a child of~$v$ in~$S$ if and only if
$\tau(u)$~is a child of $\tau(v)$ in~$T$.
\end{itemize}
\end{definition}
The nodes of a totally ordered tree~$T$ with $n$~nodes 
can be (uniquely) enumerated as 
$u_0 \prec u_1 \prec u_2 \prec \cdots \prec  u_{n-1}$.
(Note that $u_0$ must be the root node.) We extend the total order to the edges by inserting edges between consecutive nodes
\begin{equation}\label{nodes-edges-order}
u_0 \prec e_1\prec u_1 \prec e_2\prec u_2 \prec \cdots \prec e_{n-1}\prec  u_{n-1},
\end{equation}
where $e_i$ is the edge connecting $u_i$ with its parent node (which can be $u_{i-1}$, but in general it is not). 

\begin{definition}\label{def:gametree}
Let $L$ be a first-order language (a {signature}) and $k\in\N$. 
A \emph{Herbrand game tree of depth~$k$} is a labeled,
totally ordered tree~$T$ with all branches of length~$k$ where
\begin{itemize}
\item[1.] The nodes of $T$ are labeled by distinct free variables.
\item[2.] The edges are labeled by terms in $L$ satisfying the condition that a term on $e$ can only contain free variables that are labels of nodes preceding~$e$ in the $\prec$~ordering.
\item[3.] For every node~$v$, the outgoing edges have different labels.
\end{itemize}
\end{definition}

The interpretation of this definition as a game is as follows. 
There are two players, Teacher and Student, 
who alternate by Teacher playing free variables and 
Student playing terms with variables being 
among the variables played by Teacher before. 
The order is given by the order of the nodes 
and edges as shown in~(\ref{nodes-edges-order}). 
Note that according to the definition, 
the first play is by Teacher to label the root node, and in subsequent rounds, 
Teacher must immediately play the end node of the edge played by Student. 
The explanation is that when Student plays the value on
an edge~$e$, the $\prec$ order means that Student has already played values
for all edges between~$e$ and the root; therefore an omniscient Teacher has 
sufficient information to play a value that falsifies the $\forall$-quantifier corresponding
to the node below~$e$ (if such a value exists).
The intuitive explanation of condition 2. is that Student can only play what she learned from Teacher. The explanation of condition 3. is that Teacher provides only one counterexample to each of Student's replies.
The intuition for Teacher playing only distinct variables is that
this is the most general possible value for Teacher to play. 
If Student has a winning strategy~$S$ that works against 
a Teacher who plays only distinct variables, then the
same strategy would also work against a Teacher who
plays terms, just by replacing the variables in 
the terms played by Student in the strategy~$S$ with the terms
that are actually played by Teacher.

These intuitions will be formalized in Appendix~\ref{sec:midsequent}, where we construct a game satisfying these conditions from any proof (this is the content of Theorem~\ref{thm:Herbrand} below).


\begin{definition}\label{def:winningStrategy}
Let $L$ be a language and $T$ a depth~$k$ totally ordered tree with nodes labeled by distinct variables $z_0,z_1,\ldots,z_{N-1}$. 
A \emph{strategy for Student} is any labeling~$S$ of the edges
that satisfies the condition of the Herbrand game tree.

Let, furthermore, a quantifier-free formula $\varphi(y_0,x_1,y_1,\dots,x_k,y_k)$ in $L$ be given. Then $S$ is a \emph{winning strategy for Student
for~$\varphi$} provided that, letting $b$ range over the branches of~$T$ and
letting $z_{b,0},t_{b,1},z_{b,1}\dts t_{b,k},z_{b,k}$ be the 
labels on the nodes and edges of~$b$, the disjunction
\begin{equation}\label{e-winn}
\bigvee_b \varphi(z_{b,0},t_{b,1},z_{b,1}\dts t_{b,k},z_{b,k}),
\end{equation}
is a tautology. We will also consider Student-Teacher games in theories. Then we say that Student's strategy is a winning strategy in theory $\cal T$ if formula (\ref{e-winn}) is provable in~$\cal T$.
\end{definition}

{\small \paragraph{Remark.} The above definition is a technical concept that we will use. The natural concept of a strategy would involve not only what Student plays, but also \emph{where} she plays. So a strategy in that sense would include the tree and the total order on the tree.}

\medskip
The next theorem gives
the KPT Student-Teacher form of Herbrand's theorem for prenex formulas by characterizing
logical validity in terms of the existence of a Herbrand game tree
with a winning strategy.

\begin{theorem}\label{thm:Herbrand}
Let $L$ be a language and let $\Phi$ be an $L$-formula
\[
\forall y_0\,\exists x_1\,\forall y_1\cdots\exists x_k\,\forall y_k\,\varphi(y_0,x_1,y_1,\dots,x_k,y_k),
\]
where $\varphi$ is quantifier free. Then $\Phi$ is logically 
valid if and only if there is a depth~$k$ Herbrand game tree~$T$
with a winning strategy for Student for~$\varphi$.
\end{theorem}

A proof of Theorem~\ref{thm:Herbrand} is given in Appendix~\ref{sec:midsequent}.

For the sake of simplicity,
Definitions \ref{def:gametree}-\ref{def:winningStrategy} and Theorem~\ref{thm:Herbrand}
have been stated only for prenex formulas with quantifiers that start and end with $\forall$.
However, these are easily modified to accommodate prenex
formulas with any 
quantifier prefix.  The necessary modifications are obvious:
\begin{enumerate}
\item[\rm a.] If the sentence starts with $\exists$, the root is not labeled.
\item[\rm b.] If the quantifier prefix ends with $\exists$,
the leaves are not labeled since Teacher plays no values
for~$y_k$. In addition, we assume w.l.o.g.\ that the ordering~$\prec$
has all leaf nodes succeeding all internal nodes. This is justified
by the fact that Student gets no response from Teacher 
after playing a bottommost edge;
thus Student can wait to play those values at the end of the game 
without affecting any other aspect of the game. (See item~5.\ in the introduction.)
\item[\rm c.] If there is a block of existential (resp., universal) quantifiers, 
then the corresponding edges (resp., nodes) are labeled by multiple terms (resp., variables).
\end{enumerate}
Theorem~\ref{thm:Herbrand} still holds with these modifications. 

There are also modifications of trees that one could use.
In particular, one can consider a tree in which not all 
branches have full depth. This makes more sense when 
one uses the game in a concrete model. This paper will not consider such trees.

\paragraph{Copycat strategies.} A copycat strategy is 
a strategy in which one of the players always replies 
with the same move as their opponent. 
In the case of Herbrand games, 
it means that Student always plays the variable Teacher just played. Thus the strategy is uniquely determined by the order of the tree. 
Note, however, that since Student may play the variable at any edge, 
not necessarily at an edge going from the node where Teacher 
just played this variable, different $\prec$-orderings
give different copycat strategies.

In the context of Herbrand's theorem, a copycat strategy 
is something very special, because in the usual 
presentation of the Herbrand theorem function symbols are used 
and terms play an important role. For languages with
no function symbols or constant symbols, 
Student has no choice of terms other
than to play some value (a variable) previously played by Teacher.
Even in this case, copycat strategies are something special,
since the copycat strategy restricts Student to playing the most recent value played
by Teacher.

For us, copycat strategies arise naturally in our proofs, and we use them
to slightly sharpen our results.

\begin{figure}[t]
\begin{center}
\begin{tabular}{|c|c|c|c|}
\hline
$b$ & $b$ & $b$ & $b$ \\
$a_1$ & $a_1$ & $a_2$ & $a_3$ \\
$b_1$ & $b_1$ & $b_2$ & $b_3$ \\
$a_{1,1}$ & $a_{1,2}$ & $a_{2,1}$ & $a_{3,1}$ \\
$b_{1,1}$ & $b_{1,2}$ & $b_{2,1}$ & $b_{3,1}$ \\
\hline
\end{tabular}
\end{center}
\caption{The matrix corresponding to the Herbrand game
tree of Figure~\ref{fig:treeExample}. \small The first, third, and fifth rows
contain moves by Teacher. The second and fourth rows contain moves
by Student. In this example,
the columns are arranged according to the 
left-to-right order of the tree. However, this is not
required.}
\end{figure}

\paragraph{Representation by matrices.} From 
Theorem~\ref{thm:Herbrand}, and its proof in Appendix~\ref{sec:midsequent}, 
we see that Herbrand game trees provide an intuitive, game-theoretic
method for justifying logical validity. 
We now introduce yet another representation that
represents labeled, totally ordered trees as matrices of terms. 
The reason is that it is easier to define such representations by first-order formulas 
than to define the Herbrand game trees.
However, Herbrand game trees still have an advantage
of representing the \emph{structure} of 
the Herbrand disjunction better than the matrix representation does.

The transformation from labeled trees to matrices is straightforward: 
the columns of the matrix are the branches of the labeled tree. 
Each row in the matrix corresponds either to a node or to an edge. 
The top row of the matrix corresponds to the first label along
branches in the tree (usually, the Teacher label on the root). 
The second row corresponds to the next label on a branch (usually 
the labels of edges from the root). In general, the $i$-th row contains
the $i$-th labels of branches in the tree.
If the root is not labeled, which is the case of a 
sentence starting with~$\exists$, we omit the first row, 
to avoid having an empty row; and similarly for the leaves.  
Note that the order of the columns in the matrix does not 
generally agree with the left-to-right order of branches 
in the Herbrand game tree. Since the order of columns is irrelevant, 
it would be more appropriate to talk about sets of 
(column) vectors, but when formalizing it we would 
still need to introduce some order because we want to 
index them by numbers.

The transition from a matrix back to a Herbrand game tree 
is straightforward: 
view the columns as paths, glue their initial parts together 
as much as possible, and add a root if needed.

To encode the order in which players played, 
we furthermore have to index the plays of Teacher. 
In general, this means that the matrix representation needs to include 
extra information
about the $\prec$-ordering.
However, in some cases, the order can be recovered 
without this indexing; in particular this is true 
for copycat strategies:
\begin{lemma}\label{lem:l-order}
    Suppose the quantifier prefix ends with~$\forall$. 
    When Student plays a copycat strategy, 
    i.e., always repeats Teacher's last played value, 
    the total order of the tree~$T$ 
    can be recovered without additional information, 
    even if we permute the columns of the matrix.
\end{lemma}
\begin{proof}
When Student plays, Teacher's next play is directly below it 
in the matrix.
When Teacher plays a variable~$y$, Student's next 
move is where $y$ occurs on an $\exists$ row. This uniquely determines the order.
\end{proof}
Note that the opposite is trivially true: for a given total order on a tree, there is a unique way to play a copycat strategy.

\begin{corollary}\label{cor-matrix}
Assuming that the quantifier prefix ends with $\forall$, if Student plays a copycat strategy, then the corresponding matrix 
\emph{without its column ordering} completely determines the game.
\end{corollary}

When the matrix corresponds to a Herbrand-game where Student plays a copycat strategy, 
then all the entries in the matrix are variables played by Teacher. 
In such a case it will be more convenient to represent the matrix as follows. 
Let $z_1,z_2,\ldots, z_N$ be the variables played by Teacher. 
Then the matrix can be represented by the function $f(i,j)$ such that the matrix is $\{z_{f(i,j)}\}$;
that is, the entry in the $(i,j)$ position of the matrix is $z_{f(i,j)}$.
The function~$f$ captures the (copycat) strategy more naturally 
because it is irrelevant which variables we use to denote the moves of Teacher.

If the last quantifier is $\exists$ and Student plays in the last row of the matrix,
Teacher cannot respond in the same column. 
There are various possibilities for how to define the game
in this situation;
the simplest is that 
Student is allowed to play in the last row only at the very end of the game, 
and then she can play on an arbitrary number of columns without 
alternating with Teacher. Intuitively it makes sense: 
by playing on the last row without a response from Teacher, 
she cannot get any new information. 
Hence she can postpone these moves to the very end. 
This corresponds to the fact that Student plays bottommost
edges only at the end of the Herbrand game in this situation
(per the discussion after Theorem~\ref{thm:Herbrand}).

With this modification the lemma trivially generalizes as follows.

\begin{lemma}\label{lem:l-order-2}
    Suppose the formula ends with $\exists$, 
    Student plays in the last row at the very end and uses a copycat strategy 
    except for the last moves on the last row. 
    Then the total order except for the order of Student's moves 
    on the last row can be recovered from the matrix without additional information.
\end{lemma}
\begin{proof}
If Teacher plays $x$ on a node~$u$, Student plays~$x$ on some edge~$e$ because she plays a copycat strategy. So the position of the Student's next move is determined. Then Teacher has to respond on the bottom node of~$e$, hence his next move is also determined.
\end{proof}

\section{Formulas hard for arbitrary game trees}\label{sec:arbitrary}

The next theorem states that every totally ordered tree~$T$ has a corresponding 
prenex formula~$\Phi$ such that any Herbrand game tree for~$\Phi$
with a winning strategy for Student is at least as complex as~$T$.
The tree~$T$ will have all leaves at depth~$k$. Since the nodes and edges
in~$T$ correspond to moves by Teacher and Student respectively, the
Student-Teacher game played on~$T$ would correspond to the question
of whether a prenex formula with $2k+1$ quantifier blocks, of the
form
\begin{equation}{\label{eq:fmlaForT}}
\forall y_0\,\exists x_1\,\forall y_1 \,
      \cdots\,\forall y_{k-1}\,\exists x_k\,\forall y_k \, 
      \varphi
\end{equation}
with $\varphi$ quantifier-free, is logically valid. 

However, we are not able to
establish the existence of a formula exactly in the form~(\ref{eq:fmlaForT})
for every depth~$k$ totally ordered tree~$T$. Instead, we produce a
formula of the form
\begin{equation}\label{eq:fmlaTbis}
\Phi:=\forall y_0\,\exists x_1\,\forall y_1 \,
      \cdots\,\forall y_{k-1}\,\exists x_k\,\forall y_k \, \exists\vec z \,
      \varphi^*
\end{equation}
with an extra block $\exists \vec z$ of existential quantifiers and with
$\varphi^*$~quantifier-free.  Since there is no universal quantifier
in the scope of the $\exists \vec z$ quantifier, Teacher does not
reply to values that Student plays for~$\vec z$.  Therefore, 
Student can wait until
the Student-Teacher game is otherwise finished before
giving values for the existentially quantified $\vec z$ variables.
The game
ends after Student finishes giving values for the variables~$\vec z$
at all of the leaves of the tree.
In general, Student may play different strings of values for the
variables~$\vec z$ at different leaves, but in the winning strategy Student will play the same string.


\begin{theorem}\label{thm:arbitrary}
Let $T$~be a totally ordered tree with all leaves at depth~$k$.
There is a logically valid formula~$\Phi$ of the form~{\rm (\ref{eq:fmlaTbis})}
with $\varphi^*$ quantifier-free, so that the following hold:
\begin{enumerate}
\item[\rm (a)] There is a depth~$k+1$ Herbrand game tree $T^*$ for $\Phi$ 
which consists of~$T$ plus one additional
edge below each leaf node of~$T$ (for Student's plays of values for the last block of 
existential quantifiers) so that
Student has a winning strategy~$S$ for~$\Phi$. Furthermore, this strategy~$S$ is copycat except for the bottom level.
\item[\rm (b)] Vice versa, suppose that $S'$ is a winning strategy for~$\Phi$ (not necessarily 
a copycat strategy)
on a Herbrand game tree~$T'$. Then there is an embedding~$\tau$ of~$T$ into~$T'$
such that the strategy~$S'$, when restricted to the subtree~$\tau(T)$, is identical to~$S$ except for the bottom part.
\end{enumerate}
\end{theorem}
The consecutive steps in $T$ do not have to be consecutive steps in $T'$; 
Student can leave $\tau(T)$ and return there later. 
The size of~$\Phi$ will be comparable to the size of~$T$. 
This is unavoidable, because of 
the number of non-isomorphic totally ordered trees of a given size. 

Theorem~\ref{thm:arbitrary} is formally a consequence of Theorem~\ref{thm:mainPsiN} 
that we prove later,
however the formula used in it and the proof are much simpler. In
particular, 
it does not need the use of a fragment of
arithmetic.

\begin{proof}
Let $T$ be a totally ordered tree with $n$ leaf nodes, all at depth~$k$. 
Student's copycat strategy is determined by~$T$. 
The formula~$\Phi$ is formulated in the first-order language
with a single non-logical symbol, a $(2k{+}1)$-ary predicate
symbol~$P(y_0, x_1, y_1,\ldots, x_k,y_k)$. 
In this language, the only terms are variables; in any
case, a copycat strategy does not use other terms even if they are present.
However, part (b) holds even if we allow arbitrary constant and function symbols.

The formula~$\Phi$ is a
prenex form of a formula
\begin{equation}\label{eq:ExAImpP}
\forall \vec z \, A \rightarrow 
  \forall y_0\,\exists x_1\,\forall y_1 \,
      \cdots\,\forall y_{k-1}\,\exists x_k\,\forall y_k \,
    P(y_0, x_1, y_1, \ldots, x_k,y_k)
\end{equation}
where $A$~is a quantifier-free formula  
and $\forall \vec z A$ says in effect 
(loosely speaking) that
the copycat strategy is a winning strategy for Student 
for proving the formula~(\ref{eq:fmlaForT}). 
Essentially, the formula is a schema of reflection; 
a schema because $P$ is only a predicate symbol, not a formula. 
This also explains why it is logically valid, 
but we will prove it formally by showing that $S$ 
is a winning strategy for Student.  

According to Corollary~\ref{cor-matrix}, the tree, its total ordering, and Student's strategy~$S$
are determined by the representing matrix, in which a column corresponds to a branch in~$T$ and a row corresponds to a node or an edge on the branch. 
Let $N$ be the number of nodes in the tree, and $n$~be the number of branches. 
Teacher always plays a (new) variable, and these variables can be enumerated as
$z_0,\ldots,z_{N-1}$.
The order in which the variables are listed is not so important,  
but we assume for definiteness that they are listed 
in the order in which they are played. Since Student plays a copycat strategy, each
entry in the matrix is one of the $z_i$'s.
Therefore the matrix can be described by assigning variables to positions in the
matrix. For this, we use a function~$f$ with $f(i,j)$ equal to the index
of the variable placed in row~$i$ and column~$j$ of the matrix.
This function is determined by Student's strategy.
It is useful to interpret this function in terms of branches, and nodes and edges on the branches.
For an odd number~$i$, with $1\leq i\leq 2k{+}1$, and $1\leq j\leq n$, $f(i,j)$~is the index of the variable played by Teacher on the $((i{+}1)/2)$th node of the $j$th~branch. For $i$ even, $f(i,j)$ is the index of the variable played by Student on the $(i/2)$th edge of the $j$th branch. 
Hence, $z_{f(1,j)},\ldots, z_{f(2k+1,j)}$ is the sequence of variables 
that appear on the nodes and edges of the $j$th branch.

Define the formula $A$ to be
\[
A(\vec z) ~ := ~ \bigvee\nolimits_{j=1}^n P(z_{f(1,j)}\dts z_{f(2k+1,j)}).
\]
It is important to interpret this formula correctly: In $z_{f(i,j)}$, the \emph{expression} $f(i,j)$ is \emph{not} the index of $z$; the index of $z$ is the \emph{number} that $f(i,j)$ represents. For example, if $z_0$ is the label of the root, then $z_{f(1,1)}$ is this variable.
Think of~$A$ as expressing that the Herbrand game played on~$T$
yields a winning position for Student, provided Student was playing a copycat strategy.  

Define~$\Phi$ to be the prenex form of~(\ref{eq:ExAImpP}):
\[
\Phi:=\forall y_0\,\exists x_1\,\forall y_1
      \cdots\forall y_{k-1}\,\exists x_k\,\forall y_k\,\exists\vec z \,
           [ A(\vec z)\rightarrow P(y_0,x_1,y_1,\dots,y_{k-1},x_k,y_k) ].
\]

We start by proving part (a) of the theorem. 
Let $T^*$ be a tree defined by adding exactly one edge below
each leaf node; this makes $T^*$ depth~$k{+}1$, except the bottom
edges do not terminate in nodes where Teacher can respond. 
In the $\prec$-ordering of~$T^*$, the newly added leaves follow 
all the nodes of~$T$. We will show that Student has a winning strategy in which she first plays the copycat strategy~$S$ following the total ordering of~$T$, and then at the end of the game, 
plays on the bottommost edges in~$T^*$ in an arbitrary order.
On each of these edges, she will play the string of variables played by Teacher in the first part of the game. Thus the string will be the same for all these edges. The resulting strategy is copycat only until the phase when Student plays on leaves. Then Student copies what Teacher has played, but everything at once. We will show that this is a winning strategy.

Let $b_0,\ldots, b_{N-1}$ be the variables played by Teacher. Then on each bottommost edge, Student plays the vector $\vec b:=(b_0,\ldots, b_{N-1})$.
Since Student follows the given copycat strategy, 
and in particular, the players follow the total order of~$T$, 
the values played on the $\ell$th branch of $T$ are $b_{f(1,\ell)},\ldots,b_{f(2k+1,\ell)}$; 
hence the values played on the $\ell$th branch of $T^*$ are
\[
b_{f(1,\ell)},\dots,b_{f(2k+1,\ell)},\vec b.
\]
If we substitute these values into the quantifier-free part of $\Phi$ we get the formula
\[
A(\vec b)\to P(b_{f(1,\ell)},\dots,b_{f(2k+1,\ell)}),
\]
which is
\[
\bigvee_{j=1}^n P(b_{f(1,j)},\dots,b_{f(2k+1,j)})\to
P(b_{f(1,\ell)},\dots,b_{f(2k+1,\ell)}).
\]
According to the definition, this defines a winning strategy for Student iff the disjunction of these formulas over all branches is a tautology; i.e., the following formula should be a tautology
\[
\bigvee_{\ell=1}^n\left(\bigvee_{j=1}^n P(b_{f(1,j)},\dots,b_{f(2k+1,j)})\to
P(b_{f(1,\ell)},\dots,b_{f(2k+1,\ell)})\right).
\]
This is, clearly, a tautology because it is equivalent to
\[
\bigvee_{j=1}^n P(b_{f(1,j)},\dots,b_{f(2k+1,j)})\to
\bigvee_{\ell=1}^n P(b_{f(1,\ell)},\dots,b_{f(2k+1,\ell)}).
\]
This establishes part~(a) of Theorem~\ref{thm:arbitrary}. 

\medskip
The next lemma will help with proving part~(b).
\begin{lemma}\label{lem:l-propositional}
Let $p_{j,\ell}$ and $q_\ell$, $\ell\in L,j\in J$, be propositional variables, not necessarily distinct. Then
\begin{equation}\label{eq:e-A}
\bigvee\nolimits_{\ell\in L} \Bigl( \Bigl (\bigvee\nolimits_{j\in J}p_{j,\ell} \Bigr) \rightarrow q_\ell \Bigr)
\end{equation}
is a tautology if and only if for some $\ell_0\in L$
\begin{equation}\label{eq:e-B}
\{p_{j,\ell_0}:j\in J\}\subseteq \{q_\ell:\ell\in L\}.
\end{equation}
\end{lemma}
\begin{proof}
Suppose that (\ref{eq:e-B}) is false for every~$\ell_0$. 
Then for each~$\ell_0$, we can pick $p_{j_{\ell_0},\ell_0}$ that is not in 
$\{q_\ell : \ell\in L\}$, set its truth value to {\em True}, 
and set the truth values of all $q_\ell$, $\ell\in L$, to {\em False}. 
This will make (\ref{eq:e-A}) false. 

The other direction (which we do not need) is easy and left to the reader.
\end{proof}

We now prove part~(b) of Theorem~\ref{thm:arbitrary}.
Consider a game tree~$T'$ for~$\Phi$ in which Student has a winning strategy
against Teacher who, according to the definition, always plays new variables.
Let $b_{i,\ell}, a_{i,\ell}$ be the values played
by Teacher and Student for the variables $y_i, x_i$ on the $\ell$-th branch (column
of the matrix representation)  with 
Student's winning strategy. Here $1\le \ell \le n'$ where $n'$ is the number 
of leaves of~$T'$, and $i\le k$. Let $c_{i,\ell}$, $i=0,\ldots,N{-}1$, be
the values played by Student in the final answer in the $\ell$-th branch 
giving values to the variables~$z_{i}$. 

The fact that Student wins the game means that the formula
\[
\bigvee\nolimits_{\ell=1}^{n'} \, \Bigl( \Bigl( \bigvee\nolimits_{j=1}^n \, 
      P(c_{f(1,j),\ell},\dots,c_{f(2k+1,j),\ell}) \Bigr)
 \rightarrow P(b_{0,\ell},a_{1,\ell},b_{1,\ell},\dots, a_{k,\ell},b_{k,\ell}) \Bigr)
\]
is a tautology,
where $n$ is still the number of leaf nodes of~$T$ and $n'$~is the number of leaf nodes of~$T'$.
By Lemma~\ref{lem:l-propositional}, this formula is a tautology if
and only if there are branches 
$\ell_0, \ell_1, \dots, \ell_n \le n'$ 
such that for each $j=1,\ldots,n$, 
\[
P(b_{0,\ell_j},a_{1,\ell_j},b_{1,\ell_j},\ldots,a_{k,\ell_j},b_{k,\ell_j})
\]
is equal to
\[
P(c_{f(1,j),\ell_0},\dots,c_{f(2k+1,j),\ell_0}).
\]
This means that labels on the $\ell_j$-th branch up to the last node before the added leaf are $c_{f(1,j),\ell_0},\dots,c_{f(2k+1,j),\ell_0}$. 
This defines a mapping from the matrix of~$T$ into the submatrix of~$T'$ 
consisting of the columns $\ell_1\dts \ell_n$ and all rows
except for the last one. This, in turn, defines a
homomorphism~$h$ from $T$ to the subtree of $T'$ consisting of 
branches $\ell_1, \ldots, \ell_n$ with their leaf edges removed. 
As a homomorphism, $h$~preserves 
the parent-child relationship in the trees and preserves 
equality between terms labeling the nodes 
and edges of the tree. Specifically, if $u$ is a child of~$v$,
then $h(u)$ is a child of~$h(v)$. And if two elements (edges or nodes)
of~$T$ have the same label~$a$, then their images under~$h$ have the same label, denoted~$h(a)$, in~$T'$. 

It will follow from the next lemma, that the
homomorphism preserves the tree ordering, but we do not assume this.
It is not immediately clear that the homomorphism is an embedding (i.e., is injective) 
because Student is not obliged to play different values for variables~$z_{i}$. 
So it is even not clear that the branches 
$\ell_0, \ell_1, \dots, \ell_n \le n'$  are different. 
However, it is not difficult to prove that the mapping must, indeed, be one-to-one. 
To this end, we prove: 

\begin{lemma}\label{lem:embedGameTree_1}
Let $T$ be a Herbrand game tree labeled according to Student's copycat strategy. Any homomorphism of~$T$
into a Herbrand game tree is injective. Thus $T$
cannot be homomorphically embedded into a smaller Herbrand tree.
\end{lemma}

\begin{proof}
We argue by contradiction. Suppose there is a homomorphism~$h$ mapping~$T$ into a
Herbrand game tree~$S$. Letting $v_i$ denote the $i$-th node 
in the $\prec$-ordering of~$T$, let $n$
be the least number such that, for some $m<n$,
$h(v_m) = h(v_n)$. 
Let $a_m$ and $a_n$ be labels of the parent edges of
$v_m$ and~$v_n$, respectively; i.e., they
are values played by Student. Since $h(v_m) = h(v_n)$,
we have $h(a_m) = h(a_n)$.
The game on~$T$
encoded by the values~$c_{i,\ell}$ is copycat; therefore,
$a_m$ and~$a_n$ are equal to the labels on
the nodes $v_{m-1}$ and~$v_{n-1}$, respectively; i.e.,
they are equal to the previously played Teacher values.
By choice of~$n$ and since Teacher values (and thus node labels) in~$S$ are distinct, the labels on $h(v_{m-1})$ and~$h(v_{n-1})$
are distinct. These labels
are $h(a_m)$ and~$h(a_n)$,
contradicting $h(a_m)=h(a_n)$.
\end{proof}

Thus, there is an embedding of the labeled tree $T$ into the labeled tree~$T'$. 
Now, recall that the ordering can be reconstructed from the labels 
when they come from a copycat strategy 
(see Corollary~\ref{cor-matrix} and Lemma~\ref{lem:l-order-2}). 
Hence, the $\prec$ ordering is also preserved by the embedding.
This finishes the proof of the theorem.
\end{proof}

\section{Small formulas requiring large Herbrand game trees}\label{sec:LargeTrees}

Theorem~\ref{thm:arbitrary} established that given an arbitrary totally ordered
tree~$T$, there is a formula~$\Phi$ for which the Herbrand game tree
is at least as complex as~$T$. The formula~$\Phi$, however, has size
approximately the same as the size of~$T$.  We now consider the question
of giving small formulas~$\Phi$ such that any Herbrand game tree is
large and complex.  By ``large'', we show that the size of the game
tree cannot be computably bounded in the size of~$\Phi$.  By ``complex'',
we mean that the $\prec$-ordering on the tree can be arbitrarily complex.
(An explicit example of this is shown in Appendix~\ref{sec:UnivTrees}.)
To make this formal, we assume that $k>0$ and that $T(r)$ is a partial computable function
that outputs (codes for) totally ordered trees. When $T(r)$ converges, it outputs
a totally ordered depth~$k$ tree, denoted~$T_r$.
We shall prove that, when $T_r$ exists, there is a
small formula~$\Phi_r$ which requires a Herbrand game tree at least as
complex as~$T_r$.  Although $T_r$ may grow faster than any
computable function of~$r$, the size of~$\Phi_r$ will be polynomially bounded by~$r$. 
Hence in general,
there is no computable function bounding the size of~$T_r$ in
terms of the size of~$r$.

Sections \ref{sec:FirstLBproof} and~\ref{sec:SecondLBproof} will give
two different proofs that there are short formulas which require large,
complex Herbrand game trees. The first one, in Section~\ref{sec:FirstLBproof},
is based on the idea of Theorem~\ref{thm:arbitrary}, including
constructing a formula $\forall \vec z\, A$ that asserts the existence of 
correct play of a copycat game.  The second proof uses a more ``brute-force''
approach where Student and Teacher are supposed to follow a path
in a directed graph defined by a relation~$E(\cdot,\cdot)$; the length
of the path cannot be generally bounded by a computable function.
Theorem~\ref{thm:arbitrary} added an extra existential quantifier block~$(\exists \vec z)$,
increasing the Herbrand game query depth from $2k{+}1$ to $2k{+}2$. 
The first proof (Section~\ref{sec:FirstLBproof}) instead adds two
quantifier blocks, increasing the Herbrand game query depth to $2k{+}3$.
The second proof (Section~\ref{sec:SecondLBproof}) uses pure first-order logic
instead of working relative to a universal theory, and adds only a single,
existential quantifier block.

\subsection{An arithmetic theory \protect{$\mathcal T$}} \label{sec:TheoryT}

For both proofs, 
it will be useful to work with a (weak) first-order
theory~$\mathcal T$ of the integers
to express properties about the
trees~$T_r$ in the range of a partial computable function.  The theory~$\mathcal T$ must
be axiomatized with a finite set of universal axioms. 
For convenience, we can take $\mathcal T$~to use the language
of the bounded arithmetic
$T^0_2(\MSP)$~\cite{Jerabek:sharplybounded}, namely,
$0$, $S$, $+$, $\cdot$, $\#$, $\le$, $| \cdot |$, $\lfloor \frac12 \cdot \rfloor$,
and~$\MSP$.  (This is the language of~$S^1_2$~\cite{Buss:bookBA} plus
the $\MSP$ symbol for ``most significant part'',
where $\MSP(x,i) := \lfloor x/2^i \rfloor$.) 
However, using the language of $T^0_2(\MSP)$
is only a convenience, not crucial.   The axioms
of~$\mathcal T$ can be taken to be a finite set of universal axioms
that describe the basic properties of the function symbols, e.g.,
as in~\cite{Jerabek:sharplybounded}, with some additional axioms and non-logical symbols
as described next.

The crucial properties of $\mathcal T$ needed for our proofs are:
\begin{itemize}
\item[\rm (1)] The language of $\mathcal T$~contains (among other symbols)
the constant~$0$,
the unary successor function~$S$ and
the unary predecessor function~$\mathit{pd}$.
This allows $\mathcal T$ to 
be universally axiomatized and imply
the seven axioms of Robinson's theory~$Q$,
since $\mathit{pd}$ allows the axiom
$x\not=0 \rightarrow \exists y(Sy=x)$ to be
expressed as $x\not=0 \rightarrow S(\mathit{pd}(x))=x$.

For $n\in \mathbb N$, the notation~$\overline n$ denotes
the canonical term $S(S(\cdots S(0)\cdots ))$ with $n$ applications of~$S$
that has value~$n$.  $\mathcal T$~also includes the $\le$~relation
and, for each $n\in \mathbb N$, $\mathcal T$~proves
\begin{equation}\label{eq:Raxiom}
\forall x \, ( x \le \overline n 
   \, \liff \,  
          x=0 \lor x = \overline 1 \lor x = \overline 2 \lor 
         \cdots \lor x = \overline n ) .
\end{equation}
Thus any model of~$\mathcal T$ contains an initial segment
which is a copy of the standard integers.

For any closed (i.e., variable-free) term~$t$, 
$\mathcal T \vdash t =\overline n$, 
where $n\in\mathbb N$ is the value of~$t$.

\item[\rm (2)] There is a term $\langle x,y \rangle$ that defines ordered pairs
and two terms $\TPr_1$ and~$\TPr_2$ so that,  
$\mathcal T$~proves
\[
 \forall x \, \forall y \, (\TPr_1(\langle x, y\rangle ) = x 
  \,\land\, \TPr_2(\langle x,y\rangle) = y) .
\]
If $x$ is not an ordered pair, we define $\pi_1(x)=\pi_2(x)=\overline 0$.
Since $\langle x, y\rangle$ is a term, we have that, for any integers $m,n$, 
$\langle m , n \rangle \in \mathbb N$, and $\mathcal T$~proves
$\langle \overline m, \overline n \rangle = \overline {\langle m ,n \rangle}$.

\item[\rm(3)]  Nested use of the pairing function $\langle \cdot, \cdot \rangle$ can be
used to code sequences $z_0,\ldots,z_{k-1}$ of arbitrary fixed length $k \in \mathbb N$.
The idea is to encode this sequence with the value
\[
\langle z_0,z_1,\ldots,z_{k-2},z_{k-1} \rangle ~:=~
   \langle z_0, \langle z_1, \langle z_2, \ldots \langle z_{k-2}, \langle z_{k-1}, 0 \rangle \rangle \cdots \rangle \rangle \rangle .
\]
The language of~$\mathcal T$~includes a binary function symbol $\Tbeta(a,z)$
and the two axioms 
\[
  \forall z \, (\Tbeta(0,z) = \TPr_1(z) ) \quad \hbox{and} \quad \forall x\, \forall z\, (\Tbeta(S(x),z) = \Tbeta(x,\TPr_2(z))) .
\]
From this, $\mathcal T$~proves, for each fixed $k>0$,
\[
\forall z_0 \,\cdots\, \forall z_{k-1} \, \exists z \, \Bigl(
   \bigwedge\nolimits_{i=0}^{k-1} \Tbeta(\overline i, z) = z_i \Bigr).
\]
Note that $\beta(x,z)$ is always defined, but for a concrete $n \in \mathbb N$, the infinite sequence defined by $\beta(x,\overline n)$ will be zero from some index on.
\item[\rm (4)] \label{page:recQf} 
Every computable predicate $R(x_1,\ldots,x_k)$ can be defined
in~$\mathcal T$ with a quantifier-free representation as follows: 
There is a quantifier-free formula $\phi_R(z, x_1,\ldots,x_k)$
so that, for all $n_1,\ldots,n_k \in \mathbb N$, 
\begin{itemize}
\item There is a unique $m\in \mathbb N$  so that
${\mathcal T} \vdash \phi_R(\overline m, \overline n_1,\ldots, \overline n_k)$,
\item For this~$m$, $\mathbb N \vDash R(n_1,\ldots,n_k)$ holds if and only if
$\TPr_1(m)=0$ (and thus if and only if 
$\mathcal T\vdash \TPr_1(\overline m)=0$,
and if and only if $\mathcal T \nvdash \lnot \TPr_1(\overline m)=0$).
\end{itemize}
\item[\rm (4$'$)] Every computably enumerable predicate $R(x_1,\ldots,x_k)$
can be defined
in~$\mathcal T$ with a quantifier-free representation as follows: 
There is a quantifier-free formula $\psi_R(z, x_1,\ldots,x_k)$
so that, for all $n_1,\ldots,n_k \in \mathbb N$, 
\begin{itemize}
\item $R(n_1,\ldots,n_k)$ holds if and only if
$\mathbb N \vDash \psi_R(m, n_1,\ldots, n_k)$
for some $m \in \mathbb N$;
and hence if and only if 
$\mathcal T \vdash \psi_R(\overline m, \overline{n_1},\ldots,\overline{n_k})$ 
for some~$m \in \mathbb N$.
\end{itemize}
\end{itemize} 
Since the pairing and projection functions can be used to
code fixed length tuples, and $\mathcal T$ extends Robinson's
theory~$Q$, the formulas $\phi_R$ and~$\psi_R$ of conditions (4) and~(4$'$) 
can be obtained with the aid of Diophantine representations
of computable predicates with polynomials over~$\mathbb N$ (or, alternatively
using formulas over the language of~$T^0_2(\MSP)$).

Note that conditions (4) and~(4$'$) are akin to numeralwise realizability of~$R$;
however, we use a
modified notion of numeralwise realizability with the inclusion of~$m$ to be able to
use quantifier-free formulas for $\phi_R$ and~$\psi_R$.

Conditions (2) and~(3) can be realized by adding additional function
symbols and their defining equations to the language of~$\mathcal T$.%
\footnote{%
\def\Ell{\mathit{Ell}}
\def\TwoEll{\mathit{TwoEll}}
For readers familiar with $T^0_2(\MSP)$: 
In $T^0_2(\MSP)$ augmented with
truncated subtraction, a pair $\langle m,n \rangle$ can be encoded
as $w = 2^{2 \ell} + m\cdot 2^\ell + n$ where $2^\ell > \max\{m, n\}$.
E.g., 
defining $\Ell(a,b) := |a+b|$ and $\TwoEll(a,b) := (a+b) \# 1$ (so
$\TwoEll(a,b) = 2^{\Ell(a,b)} > \max\{a,b\}$) and then
$\langle a,b \rangle := \TwoEll(a,b)^2 + \TwoEll(a,b)\cdot a + b$.
The projection functions can be defined by letting
$\Ell'(w) := \lfloor |w|/2 \rfloor$ and
$\TwoEll'(w) := \MSP(\lfloor w/2\rfloor, \Ell'(w)) \# 1$ (so 
$\TwoEll'(w) = 2^{\Ell'(w)})$ and then defining
$\TPr_1(w) := \MSP(w,\Ell'(w)) - \TwoEll'(w)$ and
$\TPr_2(w) := w - \MSP(w, \Ell'(w))\cdot \TwoEll'(w)$.
}

The language of~$\mathcal T$ needs to be rich enough so that there
are {\em quantifier-free} formulas and terms that express properties of
the tree~$T_r$ generated by the partial computable function $T(r)$.
The output of~$T(r)$ (when defined) is an integer~$w$
such that $w$~encodes the totally ordered tree~$T_r$ in a canonical form.
Claim~\ref{clm:qfMethod} below will describe how to pad $w=T(r)$ with 
additional information so that the following
relations and functions are definable in~$\mathcal T$ by terms and by 
quantifier-free formulas.  The padded form of $w=T(r)$ is denoted~$w_r$
and the mapping $r \mapsto w_r$ is also partial computable.
\begin{itemize}
\item The relation $W(r,w)$ holds if and only if $w = w_r$.
\item The number of nodes of $T_r$ is equal to $N_r := N(w_r)$.
\item The number of leaves of~$T_r$ is equal to $n_r := n(w_r)$.
\item The ``parent'' relation on the nodes of~$T_r$ is defined
by letting $\Tparent(i,j,w_r)$ be true exactly
if $1\le j < i \le N_r$ and
node~$u_j$ of~$T_r$ is the parent of node~$u_i$ in~$T_r$.
Here $u_i$ means the $i$-th node in the $\prec$-ordering of
the nodes of~$T_r$, for $1\le i \le N_r$.
\end{itemize}

Note that in $\cal T$ we can represent any concrete finite structure because we can code sequences. Thus we can define each $T_r$ and argue about its properties, strategies, etc. In particular, for every fixed $r$, we can prove that $f(\overline r,i,j)$ defines a copycat strategy on (the formalization of) $T_r$.

\begin{claim}\label{clm:qfMethod}
Conditions {\rm (1)-(4$'$)} above are
sufficient to ensure that, with a suitable way of encoding the trees~$T_r$
with a $w_r \in \mathbb N$, the theory~$\mathcal T$ has quantifier-free formulas
defining $W(\cdots)$ and $\Tparent(\cdots)$, 
and terms that define the functions $N(w_r)$
and~$n(w_r)$.
\end{claim}

In particular, with the ``suitable'' encoding~$w_r$ of the Claim, 
the graph of the function $r\mapsto w_r$ is defined
by the quantifier-free $W(r, w_r)$ and hence the function $r \mapsto w_r$ is partial computable.
\begin{proof}
Let $w_r^-$ denote~$T(r)$ (when defined). Since $T(\cdot)$~is partial computable,
the set~$G_T$ of pairs $(r, w_r^-)$ is computably enumerable. By (4$'$), there is
a quantifier-free $\psi_{G_T}(m,r,w)$ so that for $r,w$, we have $(r,w) \in G_T$
if and only if there is some $m\in \mathbb N$ such that $\psi_{G_T}(m,r,w)$ holds.
Likewise, restricting $r$'s to be in the domain of~$T$,
the set~$G_N$ of pairs $(r,N_r)$ and the set~$G_n$ of pairs $(r,n_r)$ 
are computably enumerable. Hence there are quantifier-free formulas $\psi_{G_N}$
and $\psi_{G_n}$ such that, for all $r,n$, 
$\psi_{G_N}(m,r,n)$ holds for some~$m\in \mathbb N$ if and only
if $n = N(w_r)$, and similarly for $G_n$ and~$n$.

For $\Tparent$, we work with a sequence encoding all the
truth values of $\Tparent(i,j,w_r)$: Consider the set~$G_\Tparent$ of pairs $(r,v)$ such that $T(r)$ exists and
such that for all $1\le i, j \le N_r$, $\Tbeta(\langle i, j \rangle, v) = 0$
if and only if $u_j$~is the parent node of~$u_i$ in~$T_r$. This set is computably enumerable;
hence there is a quantifier-free formula $\psi_\Tparent(m, r, v)$ such that
$(r,v)$ is in this set of pairs if and only if, for some~$m$, $\psi_\Tparent(m,r,v)$ holds.

For values~$r$ such that $w_r^- = T(r)$ exists, we wish to define $w_r$ as a sequence
of four pairs:
\begin{equation}\label{eq:wDefn}
\langle ~ \langle m_W, w_r^- \rangle, \, \,
   \langle m_N, N_r \rangle, \, \,
   \langle m_n, n_r \rangle, \, \,
   \langle m_\Tparent, v_\Tparent \, \rangle ~ \rangle.
\end{equation}
The condition that $w=w_r$, namely that $w$ has this form, can be expressed as the conjunction
\begin{equation}\label{eq:wrCondition}
\psi_{G_T}(m_W, r, w_r^-) \,\land\, 
    \psi_{G_N}(m_N,r,N_r) \,\land\, 
    \psi_{G_n}(m_n,r, n_r) \,\land\, 
    \psi_\Tparent(m_\Tparent, r, v_\Tparent) .
\end{equation}
Since the constituents $m_N, w^-_r, m_n, \ldots$ of~$w_r$ can be extracted with 
the terms using $\TPr_1, \TPr_2, \Tbeta$, the property that $w$ has the correct form~(\ref{eq:wDefn})
satisfying~(\ref{eq:wrCondition}) can
be expressed as a quantifier-free formula $W(r, w)$.  As written, the $w_r$ so defined
by $(r, w_r)$
may not be unique. This can be rectified by taking the least such value~$w$: this is a standard construction and we omit the details here.

Suppose $W(r,w_r)$ holds. The functions $N(w_r)$ and~$n(w_r)$~can be defined as
the terms $\TPr_2(\Tbeta(\overline 1, w_r))$ and $\TPr_2(\Tbeta(\overline 2, w_r))$.
Similarly, $\Tparent(i,j,w_r)$
can be defined as the quantifier-free formula
\[
1\le j < i \le N(w_r) \land \Tbeta(\langle i,j \rangle, \TPr_2(\Tbeta(\overline 3,w_r)))=0 . \qedhere
\]
\end{proof}

In the next sections, we state and prove the two theorems about the complexity of Herbrand games. Both theorems use extensions of $\mathcal T$ with new function and predicate symbols designed for particular sequences of totally ordered trees~$\{T_r\}$.

In the first theorem, the theory~$\mathcal T$ is augmented
to form a theory~$\ThTPf$ which consists of~$\mathcal T$
plus a $(2k{+}1)$-ary predicate~$P(y_0,x_1,y_1,\ldots, x_k, y_k)$,
a ternary function $f(r,x,y)$ representing trees in a given computable sequence~$\{T_r\}$,
and a ``Skolemizing'' function~$f^*$ that allows $\ThTPf$ to be universal. 
Recall that a Herbrand game tree in which Student plays a copycat strategy is fully determined 
by a function $f:[2k+1]\times[n]\to Z$, where $k$ is the depth of the tree, 
$n$~is the number of leaves, and 
$Z$~is a sequence of free variables of length~$N$, where $N$ is the number of nodes. 
In~$\ThTPf$, we will formalize it as a function of three variables 
$f(r,i,j)$ mapping $[2k+1]\times[n]$ to the interval $[0,N-1]$; the value of~$f$
is equal to the index of a free variable. 
The ``free variables'' will be a string of arbitrary elements 
encoded by an element~$z$. 

Since $\{T_r\}$ is computable, $f$~is also computable.  Thus, by
condition~(4$'$) on page~\pageref{page:recQf}, there is a 
quantifier-free formula~$\psi_f$ such that
\[
f(r,x,y)=z ~\Leftrightarrow~ \exists m\, \psi_f(m,r,x,y,z) .
\]
Accordingly, the theory~$\ThTPf$ is defined as $\mathcal T$ plus two axioms
\[
\forall r \, \forall x\, \forall y\, \forall z\, 
      [ f(r,x,y)=z ~\liff~ \psi_f(f^*(r,x,y), r, x, y, z) ]
\]
and 
\[
\forall m\, \forall r \, \forall x\, \forall y\, \forall z\, 
      [ \psi_f(m,r,x,y,z) ~\rightarrow~ \psi_f(f^*(r,x,y),r,x,y,z) ].
\]
Here $f^*$ is a new 3-ary function symbol that skolemizes the 
$\exists m$ quantifier for~$\psi_f$.
Note $\ThTPf$ is universal.

By the properties of~$\mathcal T$ and since $\psi_f$ is
quantifier-free, for all $r,m,n,p\in\N$,
\[
\ThTPf\vdash\ f(\overline r,\overline m,\overline n)=\overline p\quad\mbox{ iff }\quad f(r,m,n)=p.
\]
There are no axioms about~$P$ in $\ThTPf$ other than the equality axioms.

Similar methods to those just used for $f(\cdot,\cdot,\cdot)$ 
allow $\ThTPf$ to have function
symbols for the functions $r\mapsto n_r = n(w_r)$
and $r\mapsto N_r = N(w_r)$. We will overload notation
and denote these functions by $n(r)$ and $N(r)$: it should
always be clear from context which $n(\cdot)$ or~$N(\cdot)$ function
is being used. 
(Section~\ref{sec:FirstLBproof} uses them as functions of~$r$;
Section~\ref{sec:SecondLBproof} uses them as functions of~$w_r$.)


The second proof uses theories~$\ThTG_r$, for $r\in \mathbb N$, that have a predicate $\TN(x)$ that defines
a number sort
satisfying the theory~$\mathcal T$, and in addition have
new predicates
$G_\ell(\cdots)$ and~$E(\cdot, \cdot)$,
a new function $\TTm(\cdot)$, and some associated axioms. $\ThTG_r$~will 
be described later in Section~\ref{sec:SecondLBproof}.

\subsection{Short formulas for large Herbrand trees I}\label{sec:FirstLBproof}

This section gives the first construction of formulas for every computable sequence of totally ordered trees. For every such sequence of trees $T_1, T_2, T_3, \ldots$, we construct a formula $\Psi(r)$ such that the formulas $\Psi(\overline r)$ have a similar relation to the trees as in Theorem~\ref{thm:arbitrary} in the previous section. The size of $\Psi(\overline r)$ increases very slowly with increasing~$r$, whereas the size of the trees~$T_r$ may grow as fast as any computable function. The idea of the proof is to use the same formula as in Theorem~\ref{thm:arbitrary}, but represent the part that depends on the tree succinctly. To this end we need a theory in which we can code a sequence by a single number. We will use the theory~$\ThTPf$ introduced in the previous section.

The key part of the formula used in Theorem~\ref{thm:arbitrary} was the disjunction
\[
\bigvee\nolimits_{j=1}^n P(z_{f(1,j)}\dts z_{f(2k+1,j)}),
\]
where $n$ was the number of branches of the tree we wanted to represent. Since we need a concise representation, we must use an existential quantifier instead of the disjunction. This quantifier will be in the antecedent of an implication. Hence it will become a universal quantifier in the prenex form of the formula.

\begin{theorem}\label{thm:firstLBproof}
For every computable sequence of totally ordered trees $\{T_r\}$, it is possible to construct a formula $\Psi(r)$ of the form 
\begin{equation}\label{eq:minusPsi_r}
\forall y_0\,\exists x_1\,\forall y_1\,\cdots\,\exists x_k\,\forall y_k\,\exists z\,\forall u\,
    \varphi(r,y_0,x_1,y_1,\ldots,x_k,y_k,z,u)
\end{equation}
such that 
the following holds for every $r\geq 1$.
\begin{enumerate}
\item[\rm (a)] There is a Herbrand game tree~$T$ of depth~$k+1$ with a winning strategy~$S$
for~$\Psi(\overline r)$ provably in~$\ThTPf$ such that the first $k$ levels of~$T$, including the total order, are identical to~$T_r$.
Consequently, $\Psi(\overline r)$ is provable in~$\ThTPf$.
\item[\rm (b)] Vice versa, suppose $T'$ is a totally ordered tree 
that admits a winning strategy~$S'$ for~$\Psi(\overline r)$ provably in~$\ThTPf$.
Then there is an embedding~$\tau$ of $T_r$ into~$T'$ as totally ordered trees.
\end{enumerate}
\end{theorem}
Recall that given a total order on a tree, there is a unique copycat strategy that uses this order. We will use such a strategy to prove~(a). The proof of~(b) will use the copycat strategy~$S$ of~$T_r$ and a model~$M$ of the theory~$\ThTPf$ to construct the embedding in~$T'$. The embedding will map~$T_r$ onto a subtree of~$T'$ where the strategy~$S'$ is copycat in the sense that it is copycat when terms are evaluated in~$M$.

In this theorem we only consider games in theories~$\ThTPf$ 
specially constructed for each computable sequence of trees~$\{T_r\}$. 
In the next subsection~\ref{sec:SecondLBproof}, we will present another 
construction that works already in pure logic. 
(It is also based on a special theory, $\ThTG$, but in that construction, 
the axioms needed are incorporated into the formula.)

\begin{proof}
The idea of the proof is, essentially, the same as in Theorem~\ref{thm:arbitrary}.
The formula~$\Psi(r)$ is a formalization of a certain form of the reflection principle, the statement that 
the existence of a Herbrand game tree with a winning strategy for Student 
implies the truth of
\begin{equation}\label{eq:e-sentence}
\forall y_0\,\exists x_1\,\forall y_1\,\cdots\,\exists x_k\,\forall y_k\,P(y_0,x_1,y_1,\dots,x_k,y_k).
\end{equation}
Now the statement is not for 
an arbitrary totally ordered tree, but only for trees in a particular sequence~$\{T_r\}$.

Let $r$ be given and consider the tree~$T_r$. Recall that $\ThTPf$ has the function symbols for the number of leaves~$n(r)$ and the number of nodes~$N(r)$. 
We also have a function symbol~$\Tbeta$ for coding sequences; $\Tbeta(i,z)$~is the $i$th element of the sequence coded by~$z$.
Furthermore, we have a function symbol $f(r,i,j)$ for the function that represents a copycat strategy in $T_r$ played in the given total order of~$T_r$.  To improve readability,
we will henceforth suppress the parameter~$r$, and write $f(i,j)$ instead of~$f(r,i,j)$.

We define $\varphi(r,y_0,x_1,y_1,\ldots,x_k,y_k,z,u)$, 
the quantifier-free part of $\Psi(r)$, by
\[
[1 {\leq} u {\leq} n(r)\wedge P(\Tbeta(f(\overline 1,u),z),\ldots,\Tbeta(f(\overline{2k{+}1},u),z))]\to
P(y_0,x_1,y_1,\ldots,x_k,y_k).
\]

To prove part (a), let $r$ be fixed, set $T$ to be~$T_r$ with one edge added to every leaf of~$T_r$, and let the total order be extended so that the new edges are after all nodes and edges of~$T_r$. Student's strategy will be to play first the copycat strategy determined by the total order on~$T_r$ and then play $z$ that encodes the values played by Teacher in the previous part of the game. Formally, she will play the same term 
$\zeta:=\langle z_0, z_1, \ldots, z_{N-2}, z_{N-1} \rangle$ on every leaf edge, where $z_0,\ldots,z_{N-1}$ are the values played by Teacher in the previous part of the game, and $N=N(r)$. 
Let $u_j$ be the variable played by Teacher on the leaf of branch~$j$ (after Student played $\zeta$ on the leaf edge connected to it). Then the resulting formula on branch~$j$ has the form
\begin{equation}\label{e-imp}
\begin{aligned}\relax
[1{\leq}u_j {\leq} n(\overline r)\wedge P(\Tbeta(f(\overline 1,u_j),\zeta),\Tbeta(f(\overline 2,u_j),\zeta),&\ldots,\Tbeta(f(\overline{2k{+}1},u_j),\zeta))]\to\\
&P(z_{f(1,j)},z_{f(2,j)},\dots,z_{f(2k+1,j)}).
\end{aligned}
\end{equation}
Note that in the antecedent, $f(\overline i,u_j)$ are terms in the language of $\ThTPf$, while in the consequent, $f(i,j)$ are only indices of variables $z_0,\ldots,z_{N-1}$. 
To prove that this strategy is winning, we need to show that the disjunction of these formulas for $j=1,\ldots,n(r)$ is provable in $\ThTPf$.

Since $\ThTPf$ proves $n(\overline r)=\overline{n(r)}$, we also know that $\ThTPf$ proves
\[
1 \leq u_j\leq n(\overline r)\ \equiv\ \bigvee_{1\le \ell\leq n(r)} u_j=\overline\ell.
\]
Thus the antecedent of (\ref{e-imp}) is equivalent to
\[
\bigvee_{1\leq \ell\leq n(r)} [u_j=\overline\ell\wedge P(\Tbeta(f(\overline 1,\overline\ell),\zeta),\ldots,\Tbeta(f(\overline{2k{+}1},\overline\ell),\zeta))].
\]
Using the properties of~$\beta$ and the definition of term~$\zeta$, we can further show that the formula is equivalent to
\[
\bigvee_{1\leq \ell\leq n(r)} [u_j=\overline\ell\wedge P(z_{f(1,\ell)},z_{f(2,\ell)},\dots,z_{f(2k+1,\ell)})].
\]
This disjunction trivially implies 
\[
\bigvee_{1\leq j\leq n(r)} P(z_{f(1,j)},z_{f(2,j)},\dots,z_{f(2k+1,j)}),
\]
hence the disjunction of formulas (\ref{e-imp}) is provable in~$\ThTPf$. This finishes the proof of part~(a).

\bigskip
We will now prove part (b). As in the proof of (a), we will closely follow the argument in the proof of Theorem~\ref{thm:arbitrary}. Let $T'$ be an arbitrary Herbrand game tree with Student's winning strategy $S'$, which does not have to be copycat. Let $n'$ be the number of leaves of~$T'$. For $\ell=1\dts n'$, let
\[
b_{0,\ell},a_{1,\ell},b_{1,\ell}\dts a_{k,\ell},b_{k,\ell},Z_\ell,u_\ell
\]
be the terms played on branch~$\ell$, where $b_{i,\ell}$ and~$u_\ell$ are variables. The assumption that $S'$ is a winning strategy means that the disjunction of formulas
\bel{e-imp2}
[1\le u_\ell {\leq} n(\overline r)\wedge P(\Tbeta(f(\overline 1,u_\ell),Z_\ell),\ldots,\Tbeta(f(\overline{2k{+}1},u_\ell),Z_\ell))]\to
P(b_{0,\ell},a_{1,\ell},b_{1,\ell}\dts a_{k,\ell},b_{k,\ell})
\ee
is provable in~$\ThTPf$. In this formula, 
the variables $u_\ell$ and~$b_{i,\ell}$ denote free variables, 
and the $a_{i,\ell}$'s are terms.
We can replace the antecedent in the same way as in part~(a) with 
\[
\bigvee_{1 \leq j\leq n(r)} [u_\ell=\overline j\wedge P(\Tbeta(f(\overline 1,\overline j),Z_\ell),\ldots,\Tbeta(f(\overline{2k{+}1},\overline j),Z_\ell))].
\]
We would like to use Lemma~\ref{lem:l-propositional}, but the situation is a bit more complicated because of Teacher's playing on the bottom nodes. Nevertheless, we will use the same principle.

We will construct a model~$M$  of theory~$\ThTPf$ based on the game tree~$T'$.  This model will
have interpretations for each of the variables $b_{i,j}$ and~$u_\ell$ played by Teacher during 
the play of a game.  The values of the variables $b_{i,j}$ are fixed ahead of
time. By compactness, we can require that for any such
variable~$b_{i_0,j_0}$, there is no term~$t$ involving 
only the rest of the variables~$b_{i,j}$ 
such that $b_{i_0,j_0} = t$ holds in~$M$.  
In general, a single $b$~value played by Teacher appears at multiple places
in the matrix, so we should be more precise: for any pair $(i_0,j_0)$,
there is no term~$t$ that involves only the free variables
$b_{i,j}$ such that $f(i,j) \not= f(i_0,j_0)$ and
$t = b_{i_0,j_0}$ is true in~$M$.
The variables~$u_\ell$ will be assigned
values in~$M$ during the course of a run of the Student-Teacher game: these will
be assigned a (standard) integer value~$j_\ell \in \mathbb N$, and thereby 
are represented by the closed term~$\overline j_\ell$.

In addition, the model~$M$ has
an interpretation for the predicate~$P$. Since there are no axioms that 
apply to~$P$ other than the equality axioms, we can interpret $P$ in $M$ arbitrarily, 
except that we have to satisfy the axioms of equality in~$M$.


Suppose, w.l.o.g., that Teacher plays on the bottom nodes in the order 
given by the indices of~$u_\ell$. In other words, the total order in~$T'$ 
on the bottom nodes is $u_1\prec u_2\prec\dots\prec u_{n'}$.  (If not,
just renumber the branches of~$T'$.)
We abbreviate the 
terms $\Tbeta(f(\overline i,\overline j), Z_\ell)$ 
as~$\tau^\ell_{i,j}$. 

The Student-Teacher game can be viewed as being played in the model~$M$, where
the variables $b_{i,j}$ and~$u_\ell$ played by Teacher and the terms played by
Student represent elements in the model~$M$.  The values of~$b_{i,j}$ are set
before the game begins.
The other variables $u_\ell$ played by Teacher will be assigned values $j_\ell \in \mathbb N$
during the course of the game. Student plays terms $a_{i,j}$ or $Z_\ell$: initially
these may involve variables~$u_{\ell'}$, but at the stage where Student plays
a term $a_{i,j}$ or $Z_{\ell}$, all of the variables~$u_{\ell'}$ appearing
in that term will have been replaced with some standard integer term~$\overline j_{\ell'}$.
Therefore, Student always plays a term in which the only free variables
are values~$b_{i',j'}$ that were played by Teacher in earlier steps.

By assumption,
Student has a winning strategy; furthermore, this strategy must work
when playing in~$M$ since $M\vDash \ThTPf$.
Our aim is to find an~$\ell_0$ with the property: 
\ben
\item[(*)]  for every $j=1\dts n$, 
$(\tau^{\ell_0}_{1,j},\ldots,\tau^{\ell_0}_{2k+1,j})$ 
is equal to some $(b_{0,\ell},a_{1,\ell},b_{1,\ell}\dts a_{k,\ell},b_{k,\ell})$ 
in the model~$M$, where $\ell\le n'$.
\een

Consider the following procedure in which we try successively $\ell_0=1,2,\dots$. Suppose $\ell_0=1$ does not have property~(*). So there is some $j_1$, $1\leq j_1\leq n$ such that $(\tau^{1}_{1,j_1},\ldots,\tau^{1}_{2k+1,j_1})$ is not equal to any $(b_{0,\ell},a_{1,\ell},b_{1,\ell}\dts a_{k,\ell},b_{k,\ell})$. Then we set $u_1:=\overline j_1$ and proceed to $\ell_0=2$, for which we do the same thing, except that now $u_1$ in the terms $a_{i,j}$ and~$Z_\ell$ 
is substituted by~$\overline j_1$. We continue in this way until we find an $\ell_0$ with property (*), or define such $j_\ell$'s and set $u_\ell=\overline j_\ell$ for all $\ell=1,\ldots, n'$. 

\begin{claim}
For every $h$ for which $j_h$ is defined, if $(\tau^{h}_{1,j_h},\ldots,\tau^{h}_{2k+1,j_h})$ is not equal  to $(b_{0,\ell},a_{1,\ell},b_{1,\ell}\dts a_{k,\ell},b_{k,\ell})$ at the step~$h$ of the procedure, then it remains unequal after all subsequent substitutions $u_h:=\overline j_h,u_{h+1}:=\overline j_{h+1},\dots$.
\end{claim}
\begin{proof}
The tuple $(\tau^{h}_{1,j_h},\ldots,\tau^{h}_{2k+1,j_h})$ can be expressed in
terms of~$Z_h$, and thus it is expressible by a
term that involves only variables played by Teacher before step~$h$ of the procedure. Therefore
it does not involve~$u_{h'}$ for $h' \ge h$, and so is not affected by the substitutions.
If $b_{k,\ell}$ was played before step~$h$, then also the
tuple $(b_{0,\ell},a_{1,\ell},b_{1,\ell}\dts a_{k,\ell},b_{k,\ell})$
is expressible by a term that involves only variables played before step~$h$,
so it also is not affected by the substitutions. In this case,
the two tuples cannot become equal. On the other hand,
if $b_{k,\ell}$ was played after step~$h$, then the term~$\tau^{h}_{2k+1,j_h}$
cannot equal~$b_{k,\ell}$, by the choice
of $b_{i,j}$'s as independent members of~$M$. 
So, again in this case, the two tuples remain distinct.
\end{proof}

With this claim, we can now show that the procedure always finds an $\ell_0$ satisfying~(*). Suppose this is not the case. Then we finish the construction of the model~$M$ by setting $P(b_{0,\ell},a_{1,\ell},b_{1,\ell}\dts a_{k,\ell},b_{k,\ell})$ false for every $\ell=1\dts n'$
and setting $P$ true otherwise. Then in~$M$,
\[
u_\ell=\overline j_\ell\wedge P(\Tbeta(f(\overline 1,\overline j_\ell),Z_\ell),\ldots,\Tbeta(f(\overline{2k{+}1},\overline j_\ell),Z_\ell))
\]
is true for all $\ell$, while the consequent of formula (\ref{e-imp2}) is false. Hence the disjunction of formulas (\ref{e-imp2}) is false. But this contradicts the fact that $S'$ is a winning strategy for Student.

The rest of the proof is the same as the proof of Theorem~\ref{thm:arbitrary}. 
Let $\ell_0$ satisfy~(*). Let the interpretation of~$Z_{\ell_0}$ in~$M$ be~$C$, 
and let $c_0,c_1\dts c_{N-1}$ be the sequence of length~$N$ defined by~$C$ using the function $\beta(x,z)$.\footnote{Recall that by our definition, every element defines an arbitrarily long sequence.} Then, for every $1 \le j \le n$, 
\[
M\models(\Tbeta(f(\overline 1,\overline j),Z_{\ell_0}),\ldots,\Tbeta(f(\overline{2k{+}1},\overline j),Z_{\ell_0}))=
(c_{f(1,j)}\dts c_{f(2k+1,j)}),
\]
and there exists $1\le \ell_j \le n'$ such that 
\[
M\models(c_{f(1,j)}\dts c_{f(2k+1,j)})=(b_{0,\ell_j},a_{1,\ell_j},b_{1,\ell_j}\dts a_{k,\ell_j},b_{k,\ell_j}).
\]
This defines a homomorphism of $T_r$ into~$T'$, which is in fact an embedding because the strategy~$S$ used on~$T_r$ is copycat (in light of Corollary~\ref{cor-matrix} and Lemma~\ref{lem:l-order-2}.
\end{proof}

\subsection{Short formulas for large Herbrand trees II}\label{sec:SecondLBproof}
We now give a second construction of formulas that require uncomputably large
Herbrand games. Unlike the previous construction, this construction
adds only a single block of existential quantifiers at depth~$k+1$. (The
previous construction added two blocks, namely $\exists \forall$, but worked
relative to~$\ThTP$ instead of in pure logic.) This second
construction differs essentially from the previous one in that it does
not seem to encode any statements about the soundness or reflection principle
for Herbrand games. Instead, it uses a ``brute-force'' construction to ensure
that Student must make a complicated sequence of queries that follow
the $\prec$-ordering of~$T_r$.

First fix a value for~$k$.  Second,
fix a partial computable function $r \mapsto T_r$
that produces totally ordered, depth~$k$ trees~$T_r$, as described earlier.
Equivalently, the set of pairs $(r,T_r)$ is computably enumerable.
We will construct formulas where the Herbrand game trees
for the $r$-th formula~$\Psi_r$ must be at 
least as complex as~$T_r$ (when $T_r$ exists).
The formulas, which are defined later in equation~(\ref{eq:PsiR}), have the form 
\begin{equation}\label{eq:PhiN}
\Psi_r ~:=~
\forall y_0\, \exists x_1\, \forall y_1 \cdots \exists x_k\, \forall y_k\, \exists w\, \exists \vec z\,
    \Psi^*(\overline r, y_0, x_1, \ldots, y_{k-1}, x_k, y_k, w, \vec z).
\end{equation}
The dependence of~$\Psi_r$ on $k$ and 
on the choice of the partial computable function $r \mapsto T_r$ 
is suppressed in the notation.

The formulas~$\Psi_r$ will be logically valid, but are defined
with the aid of a first-order theory~$\ThTG_r$ that is based on~$\mathcal T$. 
$\ThTG_r$~has a unary predicate~$\TN$ that defines a ``sort'' of integers; namely,
$\TN(x)$~is intended to mean that $x$ is an integer.  The elements of integer sort will
satisfy the axioms of~$\mathcal T$.  In addition to~$\TN$,
$\ThTG_r$~has $k+3$ new non-logical symbols $E$, $\TTm$, and $G_\ell$ (for $0\le \ell\le k$)
with $k+1$ new (universal) axioms.

The intuition for the lower bound is roughly as follows. 
The theory~$\ThTG_r$ has a new binary relation symbol $E(x, y)$.
Assume that
for every~$x$ there is a~$y$ such that $E(x,y)$.\footnote{We could
omit the symbol $E(\cdot,\cdot)$ from the language
and define $E(x,y)$ as
$S(\TTm(x)) = \TTm(y)$.  However, we prefer to keep $E(\cdot,\cdot)$ as a predicate for the intuition of following edges.}
Recall $N_r = N(w_r)$ is the
number of nodes in the tree~$T_r$. Teacher wishes to
guide Student to find values $i_1, i_2, i_3, \ldots, i_{N_r}$ so 
that $E(0,i_1)$ and $E(i_j, i_{j+1})$ for all~$j$.  The 
Student's goal is to find the value~$i_{N_r}$. Teacher is willing to provide
the value~$i_{j+1}$ if Student plays the value~$i_j$, but
{\em only if} Student has used the copycat strategy
and made her plays at the edges of~$T_r$
following the $\prec$-ordering of~$T_r$. If Student plays out-of-order or
does not use the copycat strategy, then Teacher provides no useful
information (e.g., just answers~0).  As long as the
Student plays the copycat strategy in the correct order, 
she wins if (1)~she plays for $N_r$ steps, eventually obtaining
the correct value for~$i_{N_r}$, or (2)~at some point Student
has played a value~$x$ but Teacher plays some value~$y$ so that
$E(x,y)$ fails. On the other hand, if Student fails to play the
copycat strategy in the correct order, then she will not reliably
win the game.

(As a warning to the reader, the term ``edge'' is used in two ways:
it can refer to either an edge in the Herbrand game tree, or
to edges in the directed graph defined by the non-logical symbols $E(x,y)$. 
These two notions
are completely different, but it should always be clear from the context
which meaning is intended.)

For fixed $r\ge 1$, $\ThTG_r$ has a unary predicate~$\TN$ picking out an
integer sort. It also has predicates $G_\ell$ and~$E$
and a function $\TTm$.
We write $\TN(x_1,\ldots,x_j)$ as an abbreviation for
$\bigwedge_{i=1}^j \TN(x_i)$. For each (universal) axiom
$\forall \vec x \varphi(\vec x)$ of~$\mathcal T$ where $\varphi$ is
quantifier-free,
$\ThTG_r$~has the axiom
\[
\forall \vec x\, [ \TN(\vec x) \rightarrow \varphi(\vec x) ].
\]
$\ThTG_r$ also has the axiom $\TN(0)$, and
for every function symbol $g$ of~$\mathcal T$, $\ThTG_r$ has the
axiom $\forall \vec x( \TN(\vec x) \rightarrow \TN(g(\vec x)))$. That is,
the number sort is closed under all function symbols.
The 
other new non-logical symbols
$G_\ell$~for $\ell=0,\ldots, k$, $\TTm$ and~$E$ are described next.
There are $k+1$ associated new axioms in~$\ThTG_r$.
\begin{itemize}
\item $E(x,y)$ is a binary relation. Informally, we think of
$E$~as the basis for a ``line-following''
algorithm where the goal is to follow edges in the directed graph
defined by the relation~$E$. In general, $E(x,y)$~does not imply $\TN(x)$ or $\TN(y)$.
\item $\TTm(b)$ is a unary function.  The intuition is that
$\TTm(b)=t$ means
$t\ge 1$~is the integer time value such that $b$~is the $t$-th
answer provided by Teacher when Student
follows a copycat strategy on the game tree based
on~$T_r$. In the cases of interest, $\TN(t)$ will hold, but
$\TN(b)$ will not.
\item For each $\ell = 0, \ldots, k$, 
\[
G_\ell( y_0, x_1, y_1, \ldots, x_\ell, y_\ell)
\]
is a $(2 \ell {+} 1)$-ary relation.
The intuitive meaning of~$G_\ell$ is that 
if Student is playing the copycat
strategy on the optimal Herbrand game tree~$T_r$, then 
$y_0, x_1, \ldots, x_\ell, y_\ell$ are correct Teacher and Student
labels on nodes and edges along some partial branch in~$T_r$.
In the cases of interest, $\TN(x_j)$ and~$\TN(y_j)$ will not hold.
\end{itemize}

For each $\ell > 0$, $\ThTG_r$~has an axiom relating $E$, $G_\ell$ and~$\TTm$, namely the universal closure of\,\footnote{This axiom is the only 
place where $\ThTG_r$ depends
on~$r$. It is needed here since the $\Tparent$ relation needs $w=w_r$ as a parameter.}
\begin{eqnarray}
\nonumber
\lefteqn{[\TN(\alpha) \land \TN(\delta) \land G_{\ell-1}(y_0,x_1,\ldots,x_{\ell-1},y_{\ell-1}) \land
    \TTm(y_{\ell-1})=\alpha} \\
& & \land \TN(w) \land W(\overline r, w) \land \Tparent(S(\delta),\alpha, w) \land \TTm(y)=\delta \land E(y,y')] 
    \label{eq:TmainGaxiom}  \\
\nonumber
& & \qquad\qquad \rightarrow 
      G_\ell(y_0,x_1,\ldots,x_{\ell-1},y_{\ell-1},y,y') \land \TTm(y')=S(\delta). 
\end{eqnarray}
obtained by universally quantifying $\alpha$, $\delta$, $w$, $y$, $y'$, and
the $x_i$'s and $y_i$'s.
The third and fourth conjuncts of the hypothesis are intended to state that the
partial branch $y_0,x_1,\ldots,y_{\ell-1}$ was created
at the $\alpha$-th step of the Student-Teacher game. 
The next five conjuncts indicate that the partial branch formed at time~$\alpha$ 
is to be extended
at time~$S(\delta) = \delta{+}1$ and 
Teacher played~$y$ at the immediately previous time
step~$\delta$. The conclusion states that the game tree is correctly
extended provided Student plays $y$ in accordance with the
copycat strategy at time $\delta{+}1$ and Teacher follows the edge
relation~$E$ for his answer~$y'$ so that $E(y,y')$ holds.

Finally, for $\ell=0$, $\ThTG_r$~has the axiom
\begin{equation}\label{eq:BaseStepAx}
\forall y_0 \, [ E(0,y_0) \rightarrow G_0(y_0) \land \TTm(y_0)=\overline 1 ] 
\end{equation}
that describes the first play by Teacher, at the root.

We now let the formula~$\Psi_r$ be:
\begin{align}
\nonumber
\forall y_0\, \exists x_1 &\, \forall y_1 \cdots 
    \exists x_k\, \forall y_k\, \exists w\, \exists \vec z\,
   (\, \TN(w) \land W(\overline r, w) \land\\
\nonumber
   & [\, ( G_k(y_0,x_1,y_1,\ldots,x_k,y_k) 
     \land \TTm(y_k)=N(w) ) \\
\label{eq:PsiR}
& \quad\lor ~  \lnot E(0,y_0) \\
\nonumber
& \quad\lor ~
\bigvee \nolimits_{i=1}^k 
   [ G_{i-1}(y_0,x_1,\ldots, x_{i-1},y_{i-1}) \\
\nonumber
& \qquad\qquad\quad \land \Tparent(S(\TTm(x_i)),\TTm(y_{i-1}),w) \land \lnot E(x_i,y_i)]\,]  \\
\nonumber
& \lor ~ \bigvee\nolimits_{\sigma \in Ax(\ThTG_r)} \lnot \sigma(\vec z) \,).
\end{align}

We let $\Psi^*_r$ denote the matrix (the quantifier-free part) of~$\Psi_r$. 
The condition $\TN(w) \land W(\overline r, w)$ expresses that $w=w_r$.
The first of the three inner disjuncts in~$\Psi^*_r$ expresses the
condition that the Student-Teacher game has successfully
reached time step $N_r=N(w_r)$ with $G_k(\cdots)$ and $\TTm(y_k)=N_r$ holding.
The other inner disjuncts of~$\Psi^*_r$ express the condition that
Teacher has failed to follow an $E$-edge in one of his responses
to Student when that was needed. The final disjunction
states that an axiom of~$\ThTG_r$ is false.
The disjunction over formulas $\sigma$ in $Ax(\ThTG_r)$ means 
$\sigma$ ranges over the quantifier-free formulas whose
universal closures form the axioms of~$\ThTG_r$. This includes
the equality axioms.

Our main theorem about $\Psi_r$ is that it requires large Herbrand
game trees, with the largeness controlled by the growth rate of the
(fast growing) partial computable function~$r \mapsto T_r$.

\begin{theorem} \label{thm:mainPsiN}
Fix $k$ as above. Let $r>0$ be such that $T_r$ exists.
Then,
\begin{itemize}
\item[(a)] $\Psi_r$ is logically valid. Furthermore it has a 
Herbrand game tree~$T$ with a winning strategy for the
Student such that $T$ is the same as~$T_r$ plus additional edges at
depth~$k+1$ (for witnessing the innermost, existential quantifiers of~$\Psi_r$).
Student's strategy is copycat except on the bottommost edges of~$T$.
\item[(b)] If $T^*$ is a Herbrand game tree for which Student 
has a winning strategy in the Student-Teacher game for~$\Psi_r$,
then $T_r$ is embeddable in~$T^*$.
\end{itemize}
\end{theorem}

\begin{proof}
We prove part~(a) first. The general idea of showing $\Psi_r$ is
valid is to follow the play of a copycat
Student-Teacher game on the tree~$T_r$ and produce
a sequent calculus proof of $\Psi_r$ with a midsequent. However, for convenience,
we will allow a certain amount of cuts (on low complexity
formulas) and will allow some existential quantifiers in what might
otherwise be the midsequent.  

We use variables to represent the values of labels on the edges and nodes in~$T_r$. 
Variables on edges are $x$-variables, and those
on nodes are $y$-variables, denoted non-uniquely $x^n_i$ and~$y^n_i$ as described next.
For $1\le n\le N_r$, let $v_n$ be the 
$n$-th node in the $\prec$-ordering of~$T_r$.
We let $x^n_i$ and $y^n_i$ denote the variables on the edges and nodes
on the path from the root of~$T_r$ to~$v_n$; namely, these
are the variables $y^n_0, x^n_1, y^n_1,\ldots, x^n_{d_n}, y^n_{d_n}$
where $d_n$ is the depth of the node~$v_n$.
Note that if $n'$ and~$n$ are on the same path in~$T_r$,
then the variables $x^n_i,y^n_i$ are the same as the variables $x^{n'}_i,y^{n'}_i$
for $i$ less than or equal to the depths of $v_n, v_{n'}$.
In addition, due to the fact that the Student plays copycat,
an $x$-variable $x^{n+1}_{d_{n+1}}$ for a value played by
Student is the same as the $y$-variable $y^n_{d_n}$ 
played by Teacher in his previous move.

For $1 \le n\le N_r$, define $F_n$ to be the formula
\begin{align*}
\lnot E(0,y^n_0)
&~\lor~ \bigvee_{0<i\le d_n} [ G_{i-1}(y^n_0,x^n_1,\ldots, x^n_{i-1},y^n_{i-1}) \\
&\qquad\qquad\quad \land \Tparent(S(\TTm(x^n_i)),\TTm(y^n_{i-1}),\overline w_r) \land \lnot E(x^n_i,y^n_i)] \\
&~\lor~ \mathrlap{\exists \vec z \bigvee_{\hbox{\scriptsize $\sigma \in Ax(\ThTG_r$})} \lnot \sigma(\vec z)}
\end{align*}
where $d_n$ is the depth of $v_n$ in~$T_r$. The intent
is that $F_n$ expresses that, in handling the $n$-th node, either the
Teacher failed to follow an $E$-edge and produce a valid $y$-variable value when he should have,
or some $\ThTG_r$-axiom is false.  

Define the formula~$S_n$ to be
\begin{equation}\label{eq:Sn}
G_{d_n}(y^n_0, x^n_1, y^n_1,\ldots, x^n_{d_n}, y^n_{d_n}) \land \TTm(y^n_{d_n})=\overline n.
\end{equation}
The intent is that $S_n$ expresses that the Student-Teacher game has correctly
reached the node~$v_n$ and produced a value~$y^n_{d_n}$ such that there is a directed
path of $E$-edges from 0 to~$y^n_{d_n}$ of length~$n$.
(``$F$'' and~``$S$'' stand for ``fail'' and ``succeed''.)

\begin{claim}\label{claim:TprovEasy}
Let $v_{n'}$ be the parent node of~$v_n$.
The following are logically valid:
\begin{itemize}
\item[\rm (i)] The sequent $F_{n'} \Rightarrow F_n$, and
\item[\rm (ii)] The sequent $S_{n'}, S_{n-1} \Rightarrow S_n, F_n$.
\end{itemize}
\end{claim}
The proof of~(i) is immediate from the definitions. This is because 
$d_{n'} = d_n{-}1$ and $x^{n'}_i, y^{n'}_i$ and $x^n_i, y^n_i$ are the
same variables, for $i < d_n$.
The proof of~(ii)
is straightforward, using
the axiom~(\ref{eq:TmainGaxiom}) of~$\ThTG_r$, the fact that $\ThTG_r$~proves
$S(\overline{n{-}1}) = \overline n$ and $\Tparent(\overline n, \overline{n'}, \overline w_r)$, 
the fact that $S_{n'}$ and $S_{n-1}$ give
$\TTm(y^{n'}_{d_{n'}})=\overline {n'}$ and
$\TTm(y^{n-1}_{d_{n-1}})=\overline {n{-}1}$,
the fact that $y^{n-1}_{d_{n-1}}$ and~$x^n_{d_n}$ are the same
variable (since the Student plays are copycat),
and the definition of~$F_n$ including
the negation of the axioms of~$\ThTG_r$ as disjuncts.

\begin{claim}\label{claim:TprovInd}
For $1 \le n \le N(w_r)$, the following sequent is logically valid:
\[
    \Rightarrow F_1 , F_2 , F_3 , \ldots , F_n , S_n.
\]
\end{claim}
The proof of the claim is by strong induction on~$n$. 
In the base case, $n=1$, the
formula $F_1 \lor S_1$ is an immediate consequence of the 
axiom~(\ref{eq:BaseStepAx}) of~$\ThTG_r$.
The induction step follows readily from the 
induction hypothesis for the parent~$n'$ of~$n$ 
and the induction hypothesis for $n{-}1$, using 
Claim~\ref{claim:TprovEasy}(ii).

\begin{claim}\label{claim:IprovIndbis}
The following is logically valid:
\[
  S_{N_r} ~\lor~ \bigvee\{ F_{n'} : \hbox{$1\le n'\le N_r$ and $v_{n'}$ is a leaf of~$T_r$} \} .
\]
\end{claim}
This claim follows readily from Claim~\ref{claim:TprovInd}
by using cuts with the sequents of Claim~\ref{claim:TprovEasy}(i) to remove 
the disjuncts~$F_{n'}$ associated with
internal, non-leaf, nodes~$v_{n'}$ of~$T_r$.

Recall the definition of~$S_n$ in~(\ref{eq:Sn}).
For $v_n$ a leaf of~$T_r$, 
let $S^*_n$ be the formula 
\[
G_k(y^n_0, x^n_1, \ldots, x^n_k, y^n_k) \land \TTm(y^n_k)=N(\overline w_r) .
\]
When $n=N(w_r)$, we have 
$\ThTG_r \vdash W(\overline r, \overline w_r) \land \overline n=N(\overline w_r)$.
Thus $\ThTG_r$~proves $S_n \rightarrow S^*_n$ for $n=N(w_r)$.
Of course, $\ThTG_r$~proves
$F_{n'} \rightarrow F_{n'} \lor S^*_{n'}$ for all leaf nodes~$v_{n'}$.
Also trivially, $\ThTG_r$ proves 
$S^*_n \rightarrow F_n \lor S^*_n$
for $n=N(w_r)$.
Thus, it follows from Claim~\ref{claim:IprovIndbis} that $\ThTG_r$ proves
\begin{equation}\label{eq:EwSnFn}
\bigvee \{ F_{n'} \lor S^*_{n'} : \hbox{$v_{n'}$ is a leaf node of $T_r$}\}
\end{equation}
Recall that $\Psi^*_r$ is defined as the matrix of~(\ref{eq:PsiR}).
To understand Student's winning strategy, 
we need to consider the formulas $\Psi^*_r(y^n_0,x^n_1,\ldots, x^n_k, y^n_k, w, \vec z)$
for $v_n$ a leaf node. 
From the fact that the formula~(\ref{eq:EwSnFn}) is $\ThTG_r$-provable,
the fact that $\ThTG_r$ proves $\TN(\overline w_r)$ and
$W(\overline r, \overline w_r)$, and the formula
$\Psi_r$ includes the disjunction $\bigvee_{\sigma\in Ax(\ThTG_r)} \lnot \sigma (\vec z)$,
we get that the sequent 
\begin{equation}\label{eq:AlmostDone}
\Rightarrow \{  
   \exists w \, \exists \vec z \, \Psi^*_r (y^n_0,x^n_1,\ldots, x^n_k, y^n_k, w, \vec z )
  : \hbox{$v_n$ is a leaf node of $T_r$} 
\}
\end{equation}
is logically valid.

We use (\ref{eq:AlmostDone}) similarly to the way the proof of Theorem~\ref{thm:midsequent} 
in appendix~\ref{sec:midsequent}
handles the midsequent. Variables $y^n_{d_n}$, $x^n_{d_n}$
are quantified in reverse order to their play
in the game.  For decreasing values of $n=N_r,N_r{-}1,\ldots, 2$,
first $y^n_{d_n}$ is
universally quantified, and then $x^n_{d_n}$ is existentially quantified.
After each
quantifier introduction, like formulas are merged into a single formula
with a contraction inference. At the end, a universal quantifier
inference is applied to $y^1_0$.
By construction, this results in a sequent
calculus proof of~$\Psi_r$ as desired, thus showing $\Psi_r$ is logically valid. 

It is clear from the above construction that the tree~$T_r$
can be used as the Herbrand game tree for~$\Psi_r$. 
That completes the proof of part~(a) of Theorem~\ref{thm:mainPsiN}.

It remains to prove part~(b).  
We need to prove that any Herbrand game tree~$T^*$ for~$\Psi_r$ must
have $T_r$ embedded in it.  

Suppose the Student-Teacher game has a winning strategy 
for Student when played on such a tree~$T^*$. Note that this
allows Student to use any strategy, not necessarily the
copycat strategy.  This strategy must work in all circumstances,
i.e.\ in any model~$M$ of~$\ThTG_r$ and 
against any Teacher strategy. Teacher now plays
members of~$M$, not free variables.
To formulate a suitable Teacher strategy,
we define a model~$M$ of~$\ThTG_r$ 
over which the Herbrand game will be played.
The theory~$\ThTG_r$ captures enough 
of the first-order theory of~$\mathbb N$ so
that the model~$M$ has a copy of the standard integers
as the interpretation~$\TN$. We add to the language
of~$\ThTG_r$ constant symbols $c_n$ for $1 \le\penalty10000 n\le\penalty10000 N(w_r)$. The $c_i$'s
will be non-integer sort in~$M$; i.e., $\TN(c_i)$ does not hold.
Therefore, by compactness, we can choose~$M$ so that no
non-trivial identity $c_i = t(\vec c)$ holds in~$M$,
except of course the trivial identities $c_i=c_i$.
In fact, we can take the universe of~$M$ to be exactly 
the standard integers plus distinct objects~$c_i$ for $1\le i \le N_r$.
This depends on the fact that in~$M$, any arithmetic function
applied to a non-numeric value returns~$0$.
Thus, in~$M$,
no $c_i$ is equal to
any term $t(c_1,\ldots, c_{i-1},c_{i+1},\ldots,c_{N_r})$ and it is impossible
for Student to give $c_i$ as an answer
if only the other~$c_j$'s are known. It also makes it impossible
for Student to play the same value~$c_i$ 
twice at the same node at the game
tree, since Student is not allowed to play the same term twice at a given node.

We also stipulate that 
$E(0,c_1)$ and, for all $1\le i < N_r = N(w_r)$, $E(c_i, c_{i+1})$ holds in~$M$. 
For all other members~$a$ of the universe of~$M$, we stipulate
that $E(a,0)$. No other values of $E(\cdot,\cdot)$ hold in~$M$.

Finally, the values of $G_i(y_0,x_1,\ldots, x_i,y_i)$
and $\TTm(y)$ need to be set.  For this, we
run a copycat strategy of Student in~$M$ using the game tree~$T_r$,
i.e., Student plays (copycat) at the nodes of~$T_r$ in $\prec$-order.
Teacher first plays~$c_1$, and thereafter
responds to Student's plays with the
values $c_2,c_3, \ldots, c_{N_r}$. 
The interpretation in~$M$ of the $\TTm$ function is that
for $1 \le i \le N_r$ and $a = c_i$, 
$\TTm(a)=i$. For 
all other values~$a$, $\TTm(a)=0$.  The interpretation
of~$G_j$ in~$M$ is set by stipulating that
if this play of the Student-Teacher game on~$T_r$ 
has Student play $c_{i-1}$ on an edge at depth~$j$ and 
thus Teacher 
answers~$c_i$ on the depth~$j$ child vertex, then $G_j(b_0,a_1,\ldots,a_j,b_j)$
holds, where $b_0, a_1,\ldots,a_j,b_j$ are the values
played by Teacher and Student on the path in~$T_r$ to where Student and
Teacher just played. (Thus, $b_0 = c_1$ and $a_j = c_{i-1}$ and $b_j = c_i$.)
Also, $G_0(c_1)$ is true in~$M$.
All other values of~$G_j$ are set false in~$M$.

To describe the embedding of~$T_r$ into~$T^*$,
we run the Student-Teacher game using Student's winning
strategy on~$T^*$. This time, we run the Student-Teacher game with
Student allowed to ask arbitrary terms in the language of~$\ThTG_r$
that make use of previous values supplied by Teacher as extra parameters.
We constrain Teacher so that Teacher, upon a value~$a$ played by Student,
will answer either with~$0$ or with the unique $c_i$ such that
$E(a, c_i)$ holds, but with the stipulation that Teacher answers~0 unless
this will immediately allow Student to win without finding~$c_{N_r}$. Specifically, Teacher
replies with~$c_i$ if the following conditions hold:
(i)~Student plays~$c_{i-1}$ at an edge~$e$ at depth~$j$
and (ii)~$G_{j-1}(b_0, a_1, b_1,\ldots, a_{j-1}, b_{j-1})$ holds
for the path in~$T^*$ leading to the parent node of the edge that holds
Student's query and (iii)~$v_{\TTm(b_{j-1})}$ is the parent node
of $v_i$ in~$T_r$.  In all other cases, Teacher responds with~$0$.

The embedding of~$T_r$ into $T^*$ is defined by
letting the $i$-th node~$v_i$ of~$T_r$
be mapped to the node in~$T^*$ where the Herbrand game 
using Student's winning strategy in~$T^*$ causes
$c_i$~to be Teacher's answer.  We claim that there must be
a unique such node in~$T^*$. If there were no such node, Student would be unable
to win the Herbrand game. This is because Teacher 
can always give answers that obey the relation~$E(\cdot, \cdot)$
as needed, so that the only way Student can win is to
successively learn $c_1,c_2, \ldots$
until eventually learning~$c_{N_r}$, and because the axioms of~$\ThTG_r$
hold in~$M$.  The fact the node must
be unique can be proved by induction on~$i$ and the fact that Student
cannot play any value~$c_j$ twice at the same node
in the game tree.

This completes the proof of part~(b) of Theorem~\ref{thm:mainPsiN}.
\end{proof}


\section{Open problems}

It would be nice to find some concrete examples of theorems whose proofs need complex trees. Finding such theorems is, certainly, not an easy task for several reasons. 
First, we do not even have good examples of theorems for which we can prove that they need long proofs; all we have are some contrived examples, although many theorems are known that have long proofs. Therefore it would be interesting to find a concrete theorem whose only known proof has a complex tree without proving that it does not have a simple one. 
The results of \cite{Jezil:Factorization,JezilTsintsilidas:StudentTeacher} 
mentioned in the introduction are the only examples that we know in which 
the structure of the Student-Teacher game plays a role.
Second, here we have only analyzed proofs based on the Herbrand theorem, while proofs in mathematics are more like proofs in the natural deduction systems, which is close to the sequent system calculus, and in Hilbert style calculi.
Third, in most mathematical theorems, the number of quantifier alternations is only two or three. If the quantifier prefix is $\forall\exists\forall$, then the depth of the trees is one and such trees are determined by a single numeric parameter. Depth-2 trees have more structure and if we also consider the total order, we get even more.

Here is an example that shows how order may be important. 
\begin{proposition}[Proposition 4.3 in \cite{Pudlak:ConsistencyGames}]
There exists a valid  $\exists\forall\exists$ sentence which does not have a Herbrand proof of the form
\[
\varphi(t_1,z_1,s_1(z_1))\ \vee\ \varphi(t_2(z_1),z_2,s_2(z_1,z_2))\ \vee\dots
\vee\ \varphi(t_\ell(z_1\dts z_{\ell-1}),z_\ell,s_\ell(z_1\dts z_{\ell-1},z_\ell)).
\]
In this formula $z_1\dts z_\ell$ are distinct variables and terms may contain only the displayed variables.
\end{proposition}
In terms of Herbrand-game trees, the proposition says that there is a sentence in the prenex form with the prefix $\exists\forall\exists$ that does not have a Student-Teacher game proof where the tree has non-splitting branches and the players play successively complete branches, i.e., they always finish a branch before going to the next. 
One can easily generalize it and construct a formula that {cannot} be proved by any depth-2 tree and players playing so that they always finish a depth-1 subtree before they proceed to the next. In other words, the sentence cannot be proved using the natural order on depth-2 trees. 
Interestingly, the proof in~\cite{Pudlak:ConsistencyGames} is non-constructive and uses G\"odel's Incompleteness Theorem, but an explicit construction of such a sentence is not known. (This result is not covered by the theorems in this paper.)

One of the remarks above suggests as an interesting research direction to study the complexity of proofs in the sequent calculus---not only the lengths of proofs, but \emph{the complexity of the structures of the proofs.} In fact, our results can be interpreted as a study of tree-like cut-free proofs. The midsequent theorem is a way to transform a cut-free proof into a kind of normal form where the first part is in propositional calculus and the second one only uses quantifier rules. Here we have shown the bottom part can be arbitrarily complex.


\appendix 

\section{Proof of KPT Game Tree properties}\label{sec:midsequent}

This appendix proves the hard direction of Theorem~\ref{thm:Herbrand}; namely, if $\Phi$ is logically
valid, then there is a Herbrand game tree with a winning
strategy for Student satisfying the conditions 1.-5.\ described
in the introduction. We present a
purely proof-theoretic proof by giving
a constructive method to convert a proof in the form given 
by Gentzen's midsequent theorem to a Student-Teacher game
with a winning strategy for Student.
A model-theoretic
proof of the correctness of the Student-Teacher game 
is sketched in Buss-Ko{\l}odziejczyk-Thapen~\cite{BKT:fragments};
it is also possible to
prove the existence of Herbrand trees from the
construction in~\cite{BKT:fragments} as well.
Another proof of the correctness of the Student-Teacher game
is given by Li-Oliveira~\cite{LiOliveira:Unprovability} who use a reduction
from the version of the Herbrand theorem proved by Buss~\cite{Buss:herbrandtheorem}
(with a correction due to McKinley~\cite{McKinley:herbrandcorrected}).  This latter
version of the Herbrand theorem is essentially the same as the original construction
by Herbrand~\cite{Herbrand:logicalwritings}.

The midsequent theorem applies to the Gentzen sequent
calculus~LK. We work with a variant of~LK
where sequents have the form $\Gamma \Rightarrow \Delta$
where $\Gamma$ and~$\Delta$ are ordered multisets of
formulas. The inferences of LK that are relevant for us
are: (1)~the quantifier inferences
\begin{equation}\label{eq:LKquantinfs}
\hbox{\Axiom$\fCenter \Gamma \cup \{A(t)\}$
\LeftLabel{$\exists$:right}
\UnaryInf$\fCenter \Gamma \cup \{\exists z \, A(z)\}$
\DisplayProof}
\qquad \qquad
\hbox{\Axiom$\fCenter \Gamma \cup \{A(b)\}$
\LeftLabel{$\forall$:right}
\UnaryInf$\fCenter \Gamma \cup \{\forall z \, A(z)\}$
\DisplayProof}
\end{equation}
where $b$~is a
free variable (called the ``eigenvariable'') that cannot appear
in the lower sequent of the $\forall$:right inference,
$z$~is a bound variable, and $t$~is a term; 
and (2)~the contraction inference
\begin{equation}\label{eq:contraction}
\hbox{
\Axiom$\fCenter \Gamma \cup \{ A, A \}$
\UnaryInf$\fCenter \Gamma \cup \{ A \}$
\DisplayProof
}
\end{equation}
where in the upper sequent, the $A$'s are any two instances of the same formula, and
one of them is removed to form the lower sequent 
(recall that $\Gamma$ is a multiset).
In the lower part,~$P$, of the proof below the midsequent (see
Theorem~\ref{thm:midsequent}) all sequents have no formulas on
the left, and all inferences 
are of the types (\ref{eq:LKquantinfs}) and~(\ref{eq:contraction}).

The explicitly indicated formulas in the upper sequents
(that is, $A(t)$, $A(b)$, and~$A$) are called the \emph{auxiliary formulas}
of the inference; the explicitly indicated formulas
in the lower sequents ($\exists z A(z)$, $\forall z A(z)$,
and~$A$) are called \emph{principal formulas} of the inference.

We can trace the occurrences of formulas in~$P$ from the midsequent 
to the endsequent
using the \emph{logical flow graph} of~\cite{Buss:kprove}, simplified
for our setting since $P$ 
contains only unary inferences. The 
logical flow graph is a directed graph. Its nodes are the
occurrences of formulas in~$P$.  For each inference, there
is a directed edge from each auxiliary formula to the principal formula
in the lower sequent.  
In addition, for each formula~$B$ in~$\Gamma$ and all appropriate~$i$,
there is a directed edge from the $i$-th occurrence of~$B$ in the~$\Gamma$
in the upper sequent to the $i$-th occurrence of~$B$ in the~$\Gamma$ in
the lower sequent.\footnote{We defined cedents to be ordered multisets precisely 
to allow defining the logical flow graph in this way.}  If $B$ and~$C$ 
are occurrences of formulas in~$P$, then $B$~is an \emph{ancestor}
of~$C$ if there is a path from $B$ to~$C$ of length $\ge 0$.
If so, we also say $C$ is a \emph{descendant} of~$B$.
Furthermore, if $B$ and~$C$ are occurrences of the same formula and $B$~is an ancestor of~$C$,
then we call $B$ a \emph{direct ancestor} of~$C$.

Let $\Phi$ be a prenex sentence 
\[
\forall \vec y_0\, \exists \vec x_1\, \forall \vec y_1 
  \cdots
  \exists \vec x_k\, \forall \vec y_k \,
  \varphi(\vec y_0, \vec x_1, \vec y_1,\ldots, \vec x_k, \vec y_k) .
\]
Each $\vec x_i$ and $\vec y_i$ is a vector of
variables, namely, $x_{i,1},\ldots, x_{i,m_i}$
and $y_{i,1},\ldots,y_{i,m'_i}$, respectively.
All the quantified variables are distinct.
We have $k\ge 1$ and
the vectors $\vec x_i$ and~$\vec y_i$ are all nonempty
except possibly $\vec y_0$ and~$\vec y_k$ are vacuous. 
W.l.o.g., there are no free variables in~$\Phi$,
as they can be included in the outermost
quantifier~$\forall \vec y_0$. 
\begin{theorem} \label{thm:midsequent}
{\rm (Midsequent Theorem~\cite{Gentzen:untersuchungen})}
Suppose $\Phi$ is a prenex sentence as above 
and that $\Phi$~is valid. There is a cut-free
LK proof of~$\fCenter\Phi$ that contains a quantifier-free
sequent~$\mathcal S$ such that every inference below~$\mathcal S$
is a contraction or a quantifier inference.
\end{theorem}
The portion of the cut-free proof from the midsequent~$\mathcal S$
to the endsequent $\fCenter \Phi$
will be denoted~$P$.
The subformula property of cut-free proofs implies 
that $\mathcal S$ has the form
\begin{equation}\label{eq:Sform}
\fCenter \bigl\{ \varphi( \vec z^\ell_0, \vec t^\ell_1, 
     \vec z^\ell_1,\ldots, \vec t^\ell_k, \vec z^\ell_k ) 
 \bigr\}_{\ell = 1}^L
\end{equation}
where the $t^\ell_{i,j}$'s are terms and
the $z^\ell_{i,j}$'s are free variables. 
Since $\mathcal S$
is quantifier-free and valid, it is tautologically valid.\footnote{Here we assume
that the
first-order language does not include equality~($=$).}

We can assume w.l.o.g.\ that the LK proof is in
free-variable normal form. Consequently, every variable 
appearing in~$\mathcal S$ is used as an eigenvariable
in the LK proof.

We prove Theorem~\ref{thm:Herbrand} by
showing how to modify the proof~$P$ to form a 
Herbrand game tree with a Student winning strategy.
In fact, $P$~can already be viewed
as a kind of Herbrand game tree that can be played by traversing $P$
upward from the endsequent $\fCenter \Phi$ to the midsequent~$\mathcal S$.
Namely, $\exists$-inferences correspond to Student proposing
a value~$t_{i,j}^\ell$ to an existentially quantified variable~$x_{i,j}$,
and $\forall$-inferences correspond to Teacher responding with
a free variable $z_{i,j}$ as a counterexample value 
for a universally quantified variable~$y_{i,j}$. Contraction
inferences correspond to nodes in the Herbrand game tree having
multiple children.  

This simple correspondence between~$P$ and a game tree gives 
only most properties of the Herbrand game described on
page~\pageref{page:HerbrandTreeProperties} in the introduction, but not
properties 2.\ and~3. Specifically,
it does not restrict Student to proposing values
for an entire block of variables $x_{i,1},\ldots,x_{i,m_i}$ at a time.
Nor does it force Teacher to then immediately respond with new
free variables as counterexamples for the next block
of universally quantified variables $y_{i,1},\ldots,y_{i,m'_i}$. It
also does not prevent Student from playing the same value
twice at the same node in the Herbrand tree.\footnote{If
we did not care about the extra properties 2.\ and~3., the
proof of Theorem~\ref{thm:Herbrand} could be finished right here.}

Therefore, we will restructure the proof~$P$ 
so that it corresponds
to the Herbrand game with these restrictions. For this, we introduce
a notion of {\em block} inferences. 
A \emph{block subformula} of~$\Phi$ is any formula of the form
\begin{equation}\label{eq:forallBlock}
\Phi^\forall_i( \vec t_i ) ~:=~ 
\forall \vec y_i \, \exists \vec x_{i+1} 
  \cdots
  \exists \vec x_k\, \forall \vec y_k \,
  \varphi(\vec z_0, \vec t_1, \vec z_1, \ldots, \vec t_{i-1}, \vec z_{i-1}, \vec t_i, \vec y_i, 
           \vec x_{i+1},\vec y_{i+1}, \cdots, \vec x_k, \vec y_k) .
\end{equation}
or
\[
\Phi^\exists_i( \vec z_i ) ~:=~ 
\exists \vec x_{i+1} \, \forall \vec y_{i+1} \, 
  \cdots
  \exists \vec x_k\, \forall \vec y_k \,
  \varphi(\vec z_0, \vec t_1, \vec z_1, \ldots, \vec t_{i-1}, \vec z_{i-1}, \vec t_i, \vec z_i,
           \vec x_{i+1},\vec y_{i+1}, \ldots, \vec x_k, \vec y_k) .
\]
These are called $\forall$-block and $\exists$-block formulas, respectively.
The bound variables $\vec x_j$'s and~$\vec y_j$'s are the same 
variables as in~$\Phi$.
The $z_i$'s and~$t_i$'s are vectors of variables and terms, respectively.
By the subformula property,
to appear in~$P$, they must be the same as the $\vec z^\ell_j$'s
and~$\vec t^\ell_j$'s in the midsequent for some~$\ell$. The formulas $\Phi^\forall_i(\vec t_i)$
and~$\Phi^\exists_i(\vec z_i)$ contain parameters $\vec z_j$'s and~$\vec t_j$'s
in addition to the indicated
$\vec t_i$ and~$\vec z_i$; however, they are suppressed in the notation.

A \emph{block inference} is an inference of the form
\[
\hbox{
\Axiom$\fCenter \Gamma, \Phi^\forall_i(\vec t_i)$
\LeftLabel{$\exists$-block}
\UnaryInf$\fCenter \Gamma, \exists \vec x_i\, \Phi^\forall_i(\vec x_i)$
\DisplayProof }
\qquad \hbox{or} \qquad
\hbox{ 
\Axiom$\fCenter \Gamma, \Phi^\exists_i(\vec z_i)$
\LeftLabel{$\forall$-block}
\UnaryInf$\fCenter \Gamma, \forall \vec y_i\, \Phi^\exists_i(\vec y_i)$
\DisplayProof}
\]
The auxiliary and principal formulas of block inferences
are block formulas. Indeed, $\exists \vec x_i\, \Phi^\forall_i(\vec x_i)$
and $\forall \vec y_i\, \Phi^\exists_i(\vec y_i)$ are the same
as $\Phi^\exists_{i-1}$ and $\Phi^\forall_i$,
respectively.  We only consider block inferences that introduce
non-vacuous (non-empty) blocks of quantifiers.

We shall iteratively form 
proofs $P_1$, $P_2$, $P_3$, \ldots, $P_r$, starting with $P_1$ equal to~$P$. 
Each $P_\ell$ satisfies the following
properties:
\begin{itemize}
\item $P_\ell$ consists of an upper part~$P'_\ell$ and a lower part~$P''_\ell$.
\item The initial sequent of~$P'_\ell$ is a
sequent~$\mathcal S_\ell$ of the form~(\ref{eq:Sform}). In fact,
$\mathcal S_\ell$ is a substitution instance of~$\mathcal S$.  The endsequent of~$P'_\ell$
is a sequent $\mathcal R_\ell$ of the form $\fCenter \Delta$ where
$\Delta$ contains only block formulas.
\item 
The initial sequent of~$P''_\ell$ is~$\mathcal R_\ell$ and its endsequent
is $\fCenter \Phi$.
\item
$P'_\ell$ contains only $\exists$:right, $\forall$:right and
contraction inferences.  
\item
$P''_\ell$~contains
only block inferences and contractions on $\exists$-block formulas.
Hence it contains only block formulas.
\item  (The {\em unique auxiliary condition}.) 
No two block inferences in~$P''_\ell$ have the same
auxiliary formula. And, no block inference in~$P''_\ell$
has the same auxiliary formula as a quantifier inference
in~$P'_\ell$.
\item (The {\em immediate-$\forall$ condition}.) 
No inference in~$P''_\ell$
has a $\forall$-block formula in~$P''_\ell$ as a side formula. That is, except
in the concluding sequent of~$P''_\ell$,
the occurrences of $\forall$-block formulas in~$P''_\ell$ must be auxiliary
formulas of $\exists$-block inferences.
\end{itemize}

Clearly $P_1 := P$ satisfies these conditions with $P'_1$ equal to all of~$P$ and
with $P''_1$ containing only the final sequent of~$P$.
Thus $\mathcal S_1$ and~$\mathcal R_1$ are the sequents
$\mathcal S$ and $\fCenter \Phi$. 
The process
stops once a $P_r$ is obtained with $P'_r$ containing
no quantifier inferences and $\mathcal S_r$~equal 
to~$\mathcal R_r$.

We need to describe how to form $P_{\ell+1}$ from~$P_\ell$.  There 
are two cases. 
First suppose $\mathcal R_\ell$ contains
a $\forall$-block formula~$\Phi^\forall_i$
of the form $\forall \vec y_i\, \Phi^\exists_i(\vec y_i)$,
see~(\ref{eq:forallBlock}).  To form~$P'_{\ell+1}$, modify $P'_\ell$ as follows. 
Let $\vec u$
be a list of $m'_i$-many new free variables: the intent is to replace
the universally quantified variables~$\vec y_i$ with~$\vec u$.
For this, replace each logical flow graph
ancestor of the formula~$\Phi^\forall_i$
of the form
\begin{equation}\label{eq:case1Form}
\forall y_{i,s}\cdots \forall y_{i,m'_i}\, 
     \Phi^\exists_i(v_1,\ldots,v_{s-1},y_{i,s},\ldots,y_{i,m'_i})
\end{equation}
with $\Phi^\exists_i(u_1,\ldots,u_{m'_i})$, and replace throughout~$P'_\ell$
all other occurrences of the variables~$v_j$ 
with the variables~$u_j$ (for $j=1,\ldots,m'_i$). 
Also remove the $\forall$:right inferences that introduced
the formulas~(\ref{eq:case1Form}) and any contraction
inferences that act on these formulas.
We obtain a proof~$P'_{\ell+1}$ of
a sequent $\mathcal R_{\ell+1}$ that is obtained from~$\mathcal R_\ell$ by
replacing the formula $\forall \vec y_i\, \Phi^\exists_i(\vec y_i)$
with multiple copies of the formula $\Phi^\exists_i(\vec u)$.
The proof~$P''_{\ell+1}$ is formed by applying $\exists$-block
contraction inferences to merge multiple occurrences of $\Phi^\exists_i(\vec u)$ in~$\mathcal R_{\ell+1}$,
then applying a $\forall$-block inference to the one remaining~$\Phi^\exists_i(\vec u)$ 
to obtain~$\mathcal R_\ell$, and then proceeding as in~$P''_\ell$.  Putting
$P'_{\ell+1}$ and~$P''_{\ell+1}$ together gives the desired
proof~$P_{\ell+1}$.  Note that the initial sequent $\mathcal S_{\ell+1}$
of~$P_{\ell+1}$ is obtained by substituting the new free variables~$u_i$
for some of the variables in~$\mathcal S_\ell$; therefore,
$\mathcal S_{\ell+1}$ is still a tautology.

The only new auxiliary formula added in the construction of~$P''_{\ell+1}$
was $\Phi^\exists_i(\vec u)$.  The variables~$\vec u$
were new, so $\Phi^\exists_i(\vec u)$ is not used elsewhere
in~$P''_{\ell+1}$, and the uses of $\Phi^\exists_i(\vec u)$
as an auxiliary formula were removed
during the construction of $P'_{\ell+1}$.
And, it is trivial
that the immediate-$\forall$ condition still holds.

The second case is when there is no $\forall$-block formula
in~$\mathcal R_\ell$; i.e., every formula in~$\mathcal R_\ell$ is an $\exists$-block
formula. Identify the {\em lowest} $\forall$-block formula,
$\forall \vec y_i\, \Phi^\exists_i(\vec t_i, \vec y_i)$,
in~$P'_\ell$ that is an ancestor of some formula
$\exists \vec x_i\, \forall \vec y_i\, \Phi^\exists_i(\vec x_i, \vec y_i)$ in~$\mathcal R_\ell$.
Then $\mathcal R_\ell$ has the form 
\begin{equation}\label{eq:endPprimeiplus}
\fCenter \Gamma,\exists \vec x_i\, \forall \vec y_i\, \Phi^\exists_i(\vec x_i, \vec y_i).
\end{equation}
By the
choice of $\forall \vec y_i \Phi^\exists_i$ as the lowest
occurrence in~$P'_\ell$ of such an ancestor, 
there are no $\forall$~inferences below the sequent containing 
$\forall \vec y_i \Phi^\exists_i(\vec t_i, \vec y_i)$ in~$P'_\ell$.
Hence, no $\forall$~inference in~$P'_\ell$ uses any variable
in~$\vec t_i$ as its eigenvariable.

Let $\mathcal V$ be the set of {\em all} occurrences of the formula 
$\forall \vec y_i\, \Phi^\exists_i(\vec t_i, \vec y_i)$ in~$P'_\ell$,
including occurrences that are not ancestors of the formula
displayed in~(\ref{eq:endPprimeiplus}).
We shall modify
$P'_\ell$ by, loosely speaking, omitting all inferences that act on
descendants of formulas in~$\mathcal V$. More precisely, an
occurrence~$A$ of a formula in~$P_\ell$ is called a $\mathcal V$-formula if $A$ is in~$\mathcal V$ or  
if every path in the logical flow graph from the initial sequent~$\mathcal S_\ell$
of~$P_\ell$ to~$A$ contains some member of~$\mathcal V$. 
Modify $P'_\ell$
by replacing every occurrence of a $\mathcal V$-formula with
$\forall \vec y_i\, \Phi^\exists_i(\vec t_i, \vec y_i)$, and omitting
inferences that act on $\mathcal V$-formulas. Also remove
contraction inferences that contract $\mathcal V$-formulas with non-$\mathcal V$-formulas
and add, as needed, a copy of the formula $\forall \vec y_i\, \Phi^\exists_i(\vec t_i, \vec y_i)$
to sequents below the removed contractions so as to maintain the property of 
having a valid sequent calculus proof. Call the resulting proof~$P^*_\ell$.

Let $\Gamma^*$ be the sequent obtained from~$\Gamma$
by removing its $\mathcal V$-formulas.
The proof~$P^*_\ell$ ends with a sequent
containing $\Gamma^*$, and one or more copies of
the formula $\forall \vec y_i\, \Phi^\exists_i(\vec t_i, \vec y_i)$, and possibly the formula
$\exists \vec x_i\, \forall \vec y_i\, \Phi^\exists_i(\vec x_i, \vec y_i)$ shown in~(\ref{eq:endPprimeiplus}) if it
is a non-$\mathcal V$-formula in~$P'_\ell$.

Form $P'_{\ell+1}$ from~$P^*_\ell$ by combining, as needed, the copies
of $\forall \vec y_i\, \Phi^\exists_i(\vec t_i, \vec y_i)$ with contraction
inferences to obtain the 
sequent 
\[
\Rightarrow \Gamma^*, \forall \vec y_i\, \Phi^\exists_i(\vec t_i, \vec y_i)\, [, \exists \vec x_i\, \forall \vec y_i\, \Phi^\exists_i(\vec x_i, \vec y_i)] . 
\]
where the formula in square brackets may or may not be 
present.
This is the endsequent $\mathcal R_{\ell+1}$ of~$P'_{\ell+1}$.

Finally, define $P''_{\ell+1}$ to start with an $\exists$-block inference
applied to $\mathcal R_{\ell+1}$ to obtain the sequent
\[
\fCenter \Gamma^*,\exists \vec x_i\, \forall \vec y_i\, \Phi^\exists_i(\vec x_i, \vec y_i) \, [, \exists \vec x_i\, \forall \vec y_i\, \Phi^\exists_i(\vec x_i, \vec y_i)]
\]
followed by contraction to remove the duplicated
formula $\exists \vec x_i\, \forall \vec y_i\, \Phi^\exists_i(\vec x_i, \vec y_i)$ if it is present.
The resulting sequent is the same as~(\ref{eq:endPprimeiplus}) except
$\Gamma$ has been replaced with~$\Gamma^*$.
The rest of $P''_{\ell+1}$ is formed from
$P''_\ell$ by removing all formulas 
in $\Gamma\setminus\Gamma^*$
and, as needed, the inferences acting on their descendants.
The immediate-$\forall$ condition is preserved.
That completes the construction of $P'_{\ell+1}$ and $P''_{\ell+1}$.

The subproof~$P''_{\ell+1}$ has
$\forall \vec y_i\, \Phi^\exists_i(\vec t_i, \vec y_i)$ as 
the auxiliary formula of the new $\exists$-block inference. 
No other
new auxiliary formula of a block inference was introduced in~$P''_{\ell+1}$. 
This formula was used at least once as an auxiliary formula in~$P'_\ell$.
Therefore, by the
induction hypothesis, this formula is not used anywhere else
as an auxiliary formula in~$P''_{\ell+1}$. In addition,
the construction of~$P^*_\ell$ ensured that there are no 
remaining uses of this formula as an auxiliary formula in~$P'_{\ell+1}$.
Thus the unique auxiliary condition still holds.

The transformation of $P_\ell$ to~$P_{\ell+1}$ always removes the introduction of 
a block formula from~$P'_\ell$ and introduces it in~$P''_{\ell+1}$ instead.
Thus the process of forming proofs~$P_\ell$ eventually stops
when there is no inference introducing a block formula in~$P'_\ell$.
When it stops, we have obtained a~$P_r$
with $P_r = P''_r$ so that $\mathcal S_r$ is equal to~$\mathcal R_r$ and is a tautology and
quantifier-free, and the endsequent of~$P_r$ is $\fCenter \Phi$.

This proof $P_r$ corresponds precisely to a Herbrand game tree. The
game is played by traversing $P_r$ from its endsequent up to~$\mathcal S_r$.
The $\forall$-block inferences correspond to moves of Teacher, who must
play completely new free variables $\vec u$ at each move. 
The $\exists$-block inferences correspond to moves of Student
playing a vector of terms~$\vec t$ for a block of existential
quantifiers. The nodes of the Herbrand game tree, namely the
moves by Teacher, correspond
to $\forall$-block inferences in~$P_r$. In addition, if
$m'_0 = 0$ and the initial
block of $\forall$ quantifiers is vacuous, then the game tree has
a root node where Teacher makes a
vacuous move so we view $P_r$~as having a final vacuous $\forall$-block inference. 
Likewise, if $m'_k$ is zero so that Teacher also
makes vacuous moves at the leaves, we view $P_r$~as containing vacuous $\forall$-block
inferences that apply to formulas in~$\mathcal S_r$.

The edges of the Herbrand game tree correspond to $\exists$-block inferences
in~$P_r$ and the nodes of the tree correspond to $\forall$-block inferences. 
A given $\exists$-block inference~$\mathcal I$ of~$P_r$ 
introducing an $\exists$-block formula $\Phi^\exists_i$ corresponds to the edge
joining the auxiliary formula (a $\forall$-block formula) of the inference~$\mathcal I$ to
the principal formula (also a $\forall$-block formula) of the (unique)
$\forall$-block inference that has a direct descendant of the formula~$\Phi^\exists_i$ as
its auxiliary formula.
The nodes in the Herbrand game tree are totally ordered according to their
order of appearance in~$P_r$ when moving upward from the endsequent.

The unique auxiliary condition implies
that Student never plays the same values at the same place in the
Herbrand tree. It also implies that Teacher never plays
twice at the same point in the tree.
The immediate-$\forall$ and unique auxiliary conditions together mean that Teacher
always plays immediately after Student in response to a play of 
Student.

\section{Universal totally ordered trees}\label{sec:UnivTrees}

This appendix constructs families of totally ordered trees of depth (at most)~$k$, 
denoted $\UnivTk p k$ and $\Univk p k$, where the parameter~$p$
denotes the phase of the construction. These trees
are universal in the sense that any totally ordered depth~$k$ tree
can be embedded into $\UnivTk p k$ and $\Univk p k$ for sufficiently
large~$p$. The trees $\UnivTk p k$ will have all leaves at depth $\le k$;
the trees~$\Univk p k$ will be formed from $\UnivTk p k$ by extending all
branches to depth~$k$.

Fix $k>1$; we write $\UnivT p$ as shorthand for $\UnivTk p k$.
The trees $\UnivT p$ are constructed cumulatively,
with $\UnivT p$ the $p$-th phase of the construction.  The 
base case is $p=0$, and the tree $\UnivT 0$ is the tree with a
single root node, which we call~$u_0$.

The totally ordered tree $\UnivT p$ for $p>0$ 
is constructed by adding a new child
node to each node in~$\UnivT {p-1}$ that is at depth~$<k$ in the tree.
Specifically, suppose $\UnivT {p-1}$~has $N$ nodes,
enumerated as $u_0 \prec \cdots \prec u_{N-1}$. Suppose also
there are $N'$~nodes in $\UnivT {p-1}$ at depth~$<k$
and enumerate them in $\prec$-order
as $u_{i_0} \prec u_{i_1} \prec \cdots \prec u_{i_{N'-1}}$.
We form $\UnivT p$ from~$\UnivT {p-1}$ by adding a new child to each
of these $N'$~nodes. The
new nodes are called $u_{N+j}$ for $j=0,1, \ldots, N'{-}1$,
and $u_{N+j}$~is a child of~$u_{i_j}$. The tree~$\UnivT p$ is
clearly a totally ordered tree with
$u_i \prec u_{i+1}$ for all $i < N+N'-1$.
It contains $\UnivT{p-1}$ as a totally ordered subtree.

Figure~\ref{fig:UnivTrees} shows examples of $\UnivT 1$,
$\UnivT 2$, $\UnivT 3$, and $\UnivT 4$ for $k=2$.

\begin{figure}[t]
\begin{center}
\parbox{0.7in}{
\begin{tikzpicture}
\node (u0) at (0,0) {$u_0$};
\node (u1) at (-1, -1) {$u_1$};
\draw (u0) -- (u1);
\end{tikzpicture}
}
\parbox{0.9in}{%
\begin{tikzpicture}
\node (u0) at (0,0) {$u_0$};
\node (u1) at (-1, -1) {$u_1$};
\node (u2) at (0,-1) {$u_2$};
\node (u3) at (-1.5,-2) {$u_3$};
\draw (u0) -- (u1);
\draw (u0) -- (u2);
\draw (u1) -- (u3);
\end{tikzpicture}
}
\parbox{1.3in}{
\begin{tikzpicture}
\node (u0) at (0,0) {$u_0$};
\node (u1) at (-1, -1) {$u_1$};
\node (u2) at (0,-1) {$u_2$};
\node (u3) at (-1.5,-2) {$u_3$};
\node (u4) at (1,-1) {$u_4$};
\node (u5) at (-1,-2) {$u_5$};
\node (u6) at (0,-2) {$u_6$};
\draw (u0) -- (u1);
\draw (u0) -- (u2);
\draw (u1) -- (u3);
\draw (u0) -- (u4);
\draw (u1) -- (u5);
\draw (u2) -- (u6);
\end{tikzpicture}
}
\parbox{1.5in}{
\begin{tikzpicture}
\node (u0) at (0,0) {$u_0$};
\node (u1) at (-1, -1) {$u_1$};
\node (u2) at (0,-1) {$u_2$};
\node (u3) at (-1.5,-2) {$u_3$};
\node (u4) at (0.75,-1) {$u_4$};
\node (u5) at (-1,-2) {$u_5$};
\node (u6) at (0,-2) {$u_6$};
\node (u7) at (1.5,-1) {$u_7$};
\node (u8) at (-0.5,-2) {$u_8$};
\node (u9) at (0.5,-2) {$u_9$};
\node (u10) at (1.25,-2) {$u_{10}$};
\draw (u0) -- (u1);
\draw (u0) -- (u2);
\draw (u1) -- (u3);
\draw (u0) -- (u4);
\draw (u1) -- (u5);
\draw (u2) -- (u6);
\draw (u0) -- (u7);
\draw (u1) -- (u8);
\draw (u2) -- (u9);
\draw (u4) -- (u10); 
\end{tikzpicture}
}
\end{center}
\caption{From left-to-right, 
the trees $\UnivTk 12$, $\UnivTk 22$, $\UnivTk 32$ and $\UnivTk 42$.  
The node~$u_i$ is the $i$-th node in the $\prec$ order.}
\label{fig:UnivTrees}
\end{figure}

\begin{figure}[t]
\begin{center}
\parbox{0.7in}{
\begin{tikzpicture}
\node (u0) at (0,0) {$v_0$};
\node (u1) at (-1, -1) {$v_1$};
\node (v2) at (-1, -2) {$v_2$};
\draw (u0) -- (u1);
\draw (u1) -- (v2);
\end{tikzpicture}
}
\parbox{0.9in}{%
\begin{tikzpicture}
\node (u0) at (0,0) {$v_0$};
\node (u1) at (-1, -1) {$v_1$};
\node (u2) at (0,-1) {$v_2$};
\node (u3) at (-1.5,-2) {$v_3$};
\node (v4) at (0,-2) {$v_4$};
\draw (u0) -- (u1);
\draw (u0) -- (u2);
\draw (u1) -- (u3);
\draw (u2) -- (v4);
\end{tikzpicture}
}
\parbox{1.3in}{
\begin{tikzpicture}
\node (u0) at (0,0) {$v_0$};
\node (u1) at (-1, -1) {$v_1$};
\node (u2) at (0,-1) {$v_2$};
\node (u3) at (-1.5,-2) {$v_3$};
\node (u4) at (1,-1) {$v_4$};
\node (u5) at (-1,-2) {$v_5$};
\node (u6) at (0,-2) {$v_6$};
\node (v7) at (1,-2) {$v_7$};
\draw (u0) -- (u1);
\draw (u0) -- (u2);
\draw (u1) -- (u3);
\draw (u0) -- (u4);
\draw (u1) -- (u5);
\draw (u2) -- (u6);
\draw (u4) -- (v7);
\end{tikzpicture}
}
\parbox{1.5in}{
\begin{tikzpicture}
\node (u0) at (0,0) {$v_0$};
\node (u1) at (-1, -1) {$v_1$};
\node (u2) at (0,-1) {$v_2$};
\node (u3) at (-1.5,-2) {$v_3$};
\node (u4) at (0.75,-1) {$v_4$};
\node (u5) at (-1,-2) {$v_5$};
\node (u6) at (0,-2) {$v_6$};
\node (u7) at (1.5,-1) {$v_7$};
\node (u8) at (-0.5,-2) {$v_8$};
\node (u9) at (0.5,-2) {$v_9$};
\node (u10) at (1.25,-2) {$v_{10}$};
\node (v11) at (2.0,-2) {$v_{11}$};
\draw (u0) -- (u1);
\draw (u0) -- (u2);
\draw (u1) -- (u3);
\draw (u0) -- (u4);
\draw (u1) -- (u5);
\draw (u2) -- (u6);
\draw (u0) -- (u7);
\draw (u1) -- (u8);
\draw (u2) -- (u9);
\draw (u4) -- (u10); 
\draw (u7) -- (v11);
\end{tikzpicture}
}
\end{center}
\caption{From left-to-right, 
the trees $\Univk 12$, $\Univk 22$, $\Univk 32$ and $\Univk 42$.  
The node~$v_i$ is the $i$-th node in the $\prec$ order.}
\label{fig:FinalUnivTrees}
\end{figure}

We define the \emph {depth} 
of~$u_i$ in $\UnivTk pk$
to equal
the distance of~$u_i$ from the root~$u_0$. Let $D^k_{d,p}$~equal
the number of nodes in~$\UnivTk p k$ at depth~$d$.  
There is a single root, so $D^k_{0,p} =\penalty10000 1$. For
phase~$p=0$, we have $D^k_{d,0} =\penalty10000 0$ for $d >\penalty10000 0$. 
For phase $p>0$ and depth $d=1,\ldots,k$, we have 
$D^k_{d,p} = D^k_{d,p-1} +\penalty10000 D^k_{d-1,p-1}$; this is because
the nodes at depth~$d$ in phase~$p$ were either present in
phase $p{-}1$ or were added in phase~$p$ as a child
of a node at depth~$d{-}1$ that was already present in phase~$p{-}1$.
Therefore, letting $\binom{i}{j}=0$ for $j>i$,
we have 
\[
D^k_{d,p} ~=~ \binom p d \quad \hbox{for $d \le k$.} 
\]
(For example, see $\UnivTk 4 2$ in Figure~\ref{fig:UnivTrees}.)
The total number of nodes in $\UnivTk p k$ is equal
to
\[
|\UnivTk p k| ~=~ \sum_{d=0}^k D^k_{d,p} 
     ~=~ \sum_{d=0}^k \binom p d .
\] 
In particular, $\UnivTk p k$ has less than $(p+1)^k = O(p^k)$ many nodes.

The tree $\Univk pk$ is formed from $\UnivTk pk$ by extending all
branches to have length exactly~$k$.  For concreteness, we
do this by considering, in $\prec$-order,
each leaf node~$u$ in~$\UnivTk pk$ which is at depth $k' < k$,
and adding a path of length $k-k'$ to extend the path through~$u$
to terminate at a leaf at depth~$k$.  The $\prec$-order is
extended to order the nodes of~$\Univk pk$ in the order that
nodes are added. (This is illustrated in the
simple case of $k=2$ in Figure~\ref{fig:FinalUnivTrees}.) In the final stage ($p$-th stage)
of forming~$\UnivTk pk$, there were exactly $D^k_{i-1,p-1}$ new leaves added
at depths $i\le k-1$. Therefore, $\Univk pk$ has 
$\sum_{i=1}^{k-1} (k{-}i) \cdot D^k_{i-1,p-1}$ more nodes than~$\UnivTk pk$.

The trees $\UnivTk pk$ and~$\Univk pk$ can be straightforwardly constructed by 
a procedure with runtime polynomial in the size of the trees.
Indeed, if $k$~is fixed, there is an algorithm which,
given $p$ as input,
constructs $\UnivTk p k$ and~$\Univk pk$ in time polynomial in~$p$. 

The next theorem shows that $\UnivTk pk$ and thereby $\Univk pk$ are ``universal''
in that all totally ordered trees can be efficiently 
embedded into them.  

\begin{theorem}
Let
$S$~be a depth~$k$ totally ordered tree with $p$ nodes. Then there is an embedding of~$S$ into~$\Univk pk$.
\end{theorem}
Note that for fixed~$k$, $\Univk pk$~is polynomial
size in the size of~$S$. 
\begin{proof}
It suffices to give an embedding of~$S$ into $\UnivTk pk$.
To start, $\tau$~maps the root of~$S$ to the
root of~$\UnivTk pk$.  Then, for successive values
of~$i>0$, the $i$-th node, $v_i$, of~$S$
is mapped to the unique node added in
the $i$-th phase of forming~$\UnivTk pk$ that satisfies 
the Definition~\ref{def:embedding} stipulation
that the parent-child relation is preserved by the
embedding.
\end{proof}

\paragraph*{AI Acknowledgement.}  AI was used in the final stages of
the paper preparation---subsequent to the first arxiv version---to 
review the paper for grammar, exposition, and
mathematical correctness. Both Claude Fable and ChatGPT Astra were used.


\nocite{Gentzen:collected}
\nocite{Herbrand:ecritslogique}

\bibliographystyle{siam}
\bibliography{logic}

\end{document}